\documentclass[11pt]{amsart}
\usepackage{amssymb, latexsym, mathrsfs,   color, amsmath }
\usepackage[all]{xy}
\usepackage[colorlinks=true, pdfstartview=FitV,linkcolor=blue,citecolor=blue,urlcolor=blue]{hyperref}
\usepackage [latin1]{inputenc}
\usepackage{chngcntr}
\counterwithin{figure}{section}

    \advance \topmargin by -\headheight
    \advance \topmargin by -\headsep

    \evensidemargin \oddsidemargin
\newtheorem {theorem}    {Theorem}[section]

\newtheorem {lemma}      [theorem]    {Lemma}
\newtheorem {corollary}  [theorem]    {Corollary}
\newtheorem {proposition}[theorem]    {Proposition}

\newtheorem {example}[theorem]    {Example}

\newcommand{\bb}{\mathbb}
\renewcommand{\rm}{\mathrm}

\newcommand{\cal}{\mathcal}

\newcommand{\GGL}{\mathrm{GL}}
\newcommand{\UU}{\mathrm{U}}
\renewcommand{\sp}{\mathrm{Sp}}
\renewcommand{\o}{\mathrm{O}}
\newcommand{\so}{\mathrm{SO}}
\newcommand{\Fq}{\bb{F}_q}
\newcommand{\fq}{(\bb{F}_q)}

\theoremstyle{definition}
\newtheorem{definition}[theorem]{Definition}
\newtheorem{remark}[theorem]{Remark}

\newcommand{\pl}{\pi_{\Lambda}}

\newcommand{\prll}{\pi_{\rho,\Lambda,\Lambda'}}

\newcommand{\prllc}{\pi_{\rho,\Lambda_1,\Lambda_1'}}

\newcommand{\pw}{{\pi}}

\newcommand{\wla}{\widetilde{\Lambda}}

\newcommand{\CD}{{\mathcal{D}}}
\newcommand{\CQ}{{\mathcal{Q}}}
\newcommand{\CJ}{{\mathcal {J}}}

\newcommand{\ee}{\epsilon_{-1}}

\numberwithin{equation}{section}

\newcommand{\fg}{\mathfrak{g}}
\newcommand{\co}{\mathcal{O}}
\newcommand{\wco}{\widetilde{\mathcal{O}}}

\begin{document}

\title{Wavefront sets and descent method for finite symplectic groups}

\date{\today}

\author[Zhifeng Peng]{Zhifeng Peng}

\address{School of Mathematical Science, Soochow University, Suzhou 310027, Jiangsu, P.R. China}
\email{zfpeng@suda.edu.cn}

\author[Zhicheng Wang]{Zhicheng Wang$^*$}

\address{School of Mathematical Science, Soochow University, Suzhou 310027, Jiangsu, P.R. China}

\email{11735009@zju.edu.cn}

\begin{abstract}
In \cite{JZ1}, D. Jiang and L. Zhang proposed a conjecture which related the wavefront sets and the descent method in the local fields case. Recently, in \cite{JLZ}, they and D. Liu define the arithmetic wavefront set of certain irreducible admissible representation $\pi$ of a classical group $G(k)$ defined over local field $k$, which is a subset of $k$-rational nilpotent orbits of the Lie algebra of $G(k)$, by the arithmetic structures of the enhanced L-parameter of $\pi$. These arithmetic structures are based on the rationality of the local Langlands correspondence and the local Gan-Gross-Prasad conjecture. They also prove that the arithmetic wavefront set is an invariant of $\pi$ (it is independent of the choice of the Whittaker datum \cite[Theorem 1.1]{JLZ}), and propose several conjectures to describe the relationship between arithmetic wavefront sets, analytic wavefront sets and algebraic wavefront sets.

In this paper we study wavefront sets of irreducible representations for finite symplectic groups and describe the relationship between wavefront sets,  descent method and finite Gan-Gross-Prasad problem. The finite fields case of Gan-Gross-Prasad problem can be calculated explicitly \cite{LW3,Wang1,Wang2}. It allows us to calculate the multiplicity of an irreducible representation in the generalised Gelfand-Graev representation corresponding to certain nilpotent orbits which are finite fields analogies of the arithmetic wavefront sets. In particular, for cuspidal representations, we give certain multiplicity one theorem and show that the finite fields analogy of arithmetic wavefront sets coincides with the wavefront sets in the sense of G. Lusztig and N. Kawanaka.
\end{abstract}

\maketitle

\section{Introduction}

Let $\overline{\mathbb{F}}_q$ be an algebraic closure of a finite field $\mathbb{F}_q$, which is of characteristic $p>2$. Consider a connected reductive algebraic group $G$ defined over $\Fq$ and its Lie algebra $\fg$, with Frobenius map $F$. Let $Z$ be the center of $G$. We will assume that $q$ is large enough such that the main theorem in \cite{S} and the Jacobson-Morozov
theorem hold.
   Let $H$ be a subgroup of $G$. Let $\pi$ (resp. $\sigma$) be a representation of $G^F$ (resp. $H^F$). We write
\[
\langle \pi,\sigma \rangle_{H^F} = \dim \mathrm{Hom}_{H^F}(\pi,\sigma ).
\]
If $G=H$, we write $\langle \pi,\sigma \rangle$ instead of $\langle \pi,\sigma \rangle_{H^F}$ for short.
By abuse of notation, for a representation $\pi$, we also denote its character by $\pi$ when no confusion arise.

Kawanaka proposed to prove Ennola's conjecture by studying \emph{Generalised Gelfand-Graev Representation} or GGGR for abbreviation. Now, GGGRs have been extremely useful in many other contexts of representation theory, see a survey \cite{Ge}.
For a nilpotent element $X \in \fg^F$ we may associate to it a ${\mathfrak{sl}}$-2 triple $\gamma=\{X,H,Y\}$,
\[
[H,X]=2X,\qquad[H,Y]=-2Y,\qquad[X,Y]=H
\]
and a representation $\Gamma_\gamma$ of $G^F$
called the associated GGGR (see \cite[Section 2]{L7} or section \ref{sec5} for details). In this paper, we focus on GGGRs and wavefront sets of irreducible representation (over $\bb{C}$) for symplectic groups over finite fields.

Let $\pi$ be an irreducible representation of $G^F$. The \emph{Kawanaka wavefront set} $\wco$
of $\pi$ is defined to be an $F$-stable nilpotent orbit $\wco$ satisfying
 \begin{itemize}
 \item[] (1)  $\langle \pi,\Gamma_\gamma\rangle\ne 0$ for some $\gamma=\{X,H,Y\}$ with $X\in \wco^F$;
\item[] (2) $\langle \pi,\Gamma_{\gamma'}\rangle\ne 0$ for some $\gamma'=\{X',H',Y'\}$ with $X'\in \wco^{\prime F}$ implies $\rm{dim}\wco'\le\rm{dim}\wco$.
\end{itemize}
If $\gamma$ is trivial, then we have $\langle \pi,\Gamma_\gamma\rangle= \rm{dim}\pi\ne 0$; hence the existence of wavefront set of an irreducible representation is obvious. Moreover, for each irreducible representation, the uniqueness of wavefront set was conjectured by Lusztig \cite{L2}, Kawanaka \cite{K2}, respectively. If this conjecture holds, then we say the wavefront set of $\pi$ is well defined. When $p$ is sufficiently large, Lusztig has shown in \cite{L7} that the wavefront sets are well defined and coincide with the ``dual'' of the unipotent
support. The existence of the unipotent support was established for $p$ good by Geck \cite{Ge1} and for any $p$ by Geck and Malle \cite{GM}.
Kawanaka also conjectured a description of the wavefront sets via Lusztig correspondence. This conjecture was proved by Lusztig in \cite{L7} and was solved in the later work of Achar and Aubert in \cite{AA} and Taylor in \cite{T} with a geometric refinement of the condition (2).
 This result implies that there is deep connection between irreducible representations of $G^F$ and the geometry of the algebraic group $G$.

In the local fields case, the notion of the wavefront set was first introduced by R. Howe in 1981 \cite{H2} for archimedean local fields.
In the p-adic case, the study of wavefront sets is related to the theory of automorphic forms (via their Fourier coefficients), and has found important applications in both areas. For example see \cite{Sh1,NPS,K2,Y,GRS,GRS2,JZ1,JZ2,JLZ}. Let $k$ be a p-adic field and $G$ be a reductive group defined over $k$.
Under the assumption that the residue characteristic is zero or ``sufficiently large'', DeBacker \cite{D2} parametrizes the set of $k$-rational nilpotent orbits of $G$ by equivalence classes of objects coming from the Bruhat-Tits building of the group. For p-adic fields, there are the following variants of wavefront sets: analytic, algebraic, unramified, and arithmetic wavefront sets. Let us introduce them briefly.

(1) {\bf Analytic wavefront set.} The p-adic analytic wavefront set of a representation $\pi$ is a collection of certain $k$-rational nilpotent orbits of $G(k)$ depending on the Harish-Chandra-Howe local character expansion \cite{H,HC,D1} of $\pi$. The analytic wavefront sets for $\GGL_n$ and for irreducible subquotients of the regular principal series for split classical groups have been calculated in \cite{MW}. However, it is extremely difficult to calculate the analytic wavefront sets in general.

(2) {\bf Algebra wavefront set.} The difficulty of calculation of analytic wavefront sets leads to consider $\overline{k}$-rational nilpotent orbits and so-called algebraic wavefront sets of $\pi$. (Kawanaka wavefront sets also focus on the stable nilpotent orbits, and algebra wavefront sets the corresponding notion for local fields.) In \cite{MW}, M\oe glin and Waldspurger define the algebraic wavefront set by means of the generalized Whittaker models. For classical groups,
algebraic wavefront sets must always be a special orbit \cite{M}, and Waldspurger \cite{Wal2,Wal3}, computed the wavefront sets of unipotent representations of $\so_{2n+1}$.

(3) {\bf Unramified wavefront sets.} Barbasch and Moy classified rational nilpotent orbits in their seminal paper \cite{BM} and provide relationship between the analytic wavefront set and the Kawanaka wavefront set. Applying their description of wavefront sets, Okada \cite{O} calculated the unramified wavefront sets which are intermediate to the analytic and algebraic wavefront sets of spherical Arthur representations of split p-adic reductive groups.

 (4) {\bf Arithmetic wavefront set.} In \cite{JZ1}, D. Jiang and L. Zhang proposed a conjecture of wavefront sets and they suggested that in the local fields case, one can obtain the wavefront set of an irreducible representation in a generic L-packet via the local descent method. Recently, D. Jiang, D. Liu and L. Zhang \cite{JLZ} define the arithmetic wavefront set of certain irreducible admissible representation $\pi$ of a classical group over local field by the arithmetic structures of the enhanced L-parameter of $\pi$. Their definition of the arithmetic wavefront set based on the rationality of the local Langlands correspondence and the local Gan-Gross-Prasad conjecture. They also prove that the arithmetic wavefront set is an invariant of $\pi$ (it is independent of the choice of the Whittaker datum \cite[Theorem 1.1]{JLZ}), and propose several conjectures to describe the relationship between the arithmetic wavefront sets, the analytic wavefront sets and the algebraic wavefront sets. To be more precise, the set of the maximal members of the arithmetic wavefront set coincides with the set of maximal members of the analytic wavefront set and the algebraic wavefront set.

 In this paper, we shall prove some cases of finite fields analogy of their conjectures by known results of the finite Gan-Gross-Prasad problem \cite{LW2,LW3,LW4,Wang1,Wang2}, and calculate Kawanaka wavefront sets for some irreducible representations.

 In \cite{GRS,Gi} D. Ginzburg, S. Rallis and D. Soudry used the descent method to calcualte the wavefront set of automorphic cuspidal representation. The method developed in \cite{GRS} implied one should be able to read the wavefront set of representations by composition of descents (cf. \cite[Conjecture 1.8]{JLZ} for the $p$-adic case). In the finite fields case, for instance, consider finite unitary groups. By Lusztig's classification of irreducible representations and Springer correspondence, one can associate a partition $\lambda$ with an irreducible unipotent representation $\pi_\lambda$ and the wavefront set of $\pi_\lambda$ is $\lambda^t$, the transpose of $\lambda$, which is equal to the partition we get by the descent method. However, one may not get the wavefront set directly in this way for finite symplectic groups (see Example \ref{ue}), which is quite different from finite unitary groups. This shows how tricky things can be for symplectic groups.

Recently, the behaviour of the wavefront set and the generalized Whittaker models under Theta correspondence was studied in \cite{AKP,GZ,LM}. In our pervious work \cite{LW3,Wang1}, we have shown the relationship between the descent method and Theta correspondence over finite fields. Motivated by the above works, in this paper, we will prove the finite fields analogy of Conjecture 1.8 in \cite{JZ1} and Conjecture 1.6 in \cite{JLZ} for certain irreducible representations which include some important cases, like cuspidal representations and unipotent representations. We give explicit description of wavefront sets for certain irreducible representations (for example cuspidal representations) and get some multiplicity one results. Moreover, we will not only read the Kawanaka wavefront sets by the descent method, but also get some information about the $F$-rational orbits corresponding to these representations. The main theorem in this paper about $F$-rational nilpotent orbits has potential to calculate the wavefront sets in the p-adic case.

 \subsection{Main results}

Let $G$ be a connected reductive algebraic group defined over $\Fq$ with Frobenius map $F$.
Let $G^*$ be the dual group of $G$ (a finite fields analogue of Langlands dual group see e.g. \cite[P.111]{C}). We still denote the Frobenius endomorphism of $G^*$ by $F$.
For a semisimple element $s \in G^{*F}$, define Lusztig series as follows:
\[
\mathcal{E}(G^F,s) = \{ \pi \in \rm{Irr}(G^F)  :  \langle \pi, R_{T^*,s}^G\rangle \ne 0\textrm{ for some }T^*\textrm{ containing }s \}
\]
where $T^*$ is an $F$-stable maximal torus of $G^*$ and $R_{T^*,s}^G$ is the Deligne-Lusztig correspondence to pair $(T^*,s)$.
And
\[
\rm{Irr}(G^F)=\coprod_{(s)}\mathcal{E}(G^F,s)
\]
where $(s)$ runs over the conjugacy classes of semisimple elements.

In this paper, we only consider the following three cases:
\begin{itemize}

\item (A) $s$ has no eigenvalues $\pm1$;

\item (B) either $s$ only has eigenvalues $1$ or $s$ only has eigenvalues $-1$;

\item (C) $\pi$ is cuspidal.
\end{itemize}

Recently, in the local field case, theta lifting was successfully used by a number of authors to understand nilpotent invariants of representations \cite{BMSZ1,BMSZ2,GZ,LM,Z}, especially for unipotent representations. In finite field case, we expect that we can get the general case from above three cases by the similar theta lifting argument in \cite{GZ,Z}. And we will study the general case in our further work by theta lifting.

From now on, we use wavefront sets instead of Kawanaka wavefront sets for simplicity. For an irreducible representation $\pi$, we can associate it with an array $\hat{\ell}{(\pi)}$ via the descent method in Section \ref{sec5.4}. In this paper, a partition is an array $\lambda=(\lambda_i)$ with $\lambda_i\ge \lambda_{i+1}$ for every $i$. We will conclude that the array $\hat{\ell}{(\pi)}$ is a partition in the Case (A) and Case (C). However, for Case (B), the array $\hat{\ell}{(\pi)}$ may not be a partition. So we need a modification, and then we will get a partition $\widetilde{\ell}{(\pi)}$ in Section \ref{sec6.2}.

Recall that the $F$-stable nilpotent orbits of symplectic group $\sp_{2n}$ are parameterized by all partitions
of $2n$ where odd numbers occur with even multiplicity. For such a partition $\lambda$, we denote the corresponding $F$-stable nilpotent orbit by $\wco_{\lambda}$.

\begin{theorem}\label{main1}
 Let $\pi\in\rm{Irr}(\sp_{2n}\fq)$. Let $\ell_{0,\psi}^{\rm{FJ}}$ be the first occurrence index of $\pi$ in the descent with respect to certain character $\psi$ of $\Fq$ (see section \ref{sec5.4} for detail).

(i) Assume that $\pi\in \mathcal{E}(G^F,s)$ is in Case (A) and $C_{G^{*F}}(s)$ is a product of general linear groups. Then $\wco_{\hat{\ell}{(\pi)}}$ is the wavefront set of $\pi$. Moreover, every part of $\hat{\ell}{(\pi)}$ is even and the first part of $\hat{\ell}{(\pi)}$ is equal to $2\ell_{0,\psi}^{\rm{FJ}}$ for any character $\psi$ of $\Fq$.

(ii) Assume that $\pi$ is in Case (B). Then $\wco_{\widetilde{\ell}{(\pi)}}$ is the wavefront set of $\pi$. Moreover, the first part of $\widetilde{\ell}{(\pi)}$ is either $2\ell_{0,\psi}^{\rm{FJ}}$ or $2\ell_{0,\psi}^{\rm{FJ}}+1$.

(iii) Assume that $\pi$ is in Case (C). Then $\wco_{\hat{\ell}{(\pi)}}$ is the wavefront set of $\pi$. Moreover, every part of $\hat{\ell}{(\pi)}$ is even and the first part of $\hat{\ell}{(\pi)}$is equal to $2\ell_{0,\psi}^{\rm{FJ}}$ with respect to certain character $\psi$ of $\Fq$.
\end{theorem}

We have explicitly described the first occurrence index in descent method in our previous work \cite{Wang1,Wang2}, and the array $\hat{\ell}{(\pi)}$ or partition $\widetilde{\ell}{(\pi)}$ and the wavefront set of $\pi$ can be calculated explicitly.

Conjecture 1.8 in \cite{JZ1} implies that for an irreducible $\pi$ in a generic L-packet, the first part of the partition corresponding to the wavefront set of $\pi$ is equal to $2\ell_{0,\psi}^{\rm{FJ}}$. Conjecture 1.6 in \cite{JLZ} implies that $\wco_{\hat{\ell}{(\pi)}}$ is the wavefront set of $\pi$.
However, by Theorem \ref{main1}, it is not always the case. As an explanation of this phenomenon, one should note that in Jiang, Liu and Zhang's conjectures, they require that $\pi$ is a tempered representation. In the local field case, the descent method may not work on certain non-tempered representations. However, in the finite field case, normally we do not care about these analytical properties. But we have some irreducible representations with similar behavior.

\begin{example}\label{ue}
Consider the trivial representation $\bf{1}$ of $\sp_2\fq$. One can check that $\cal{D}^{\rm{FJ}}_\psi(\bf{1})=\sigma+\sigma'$ where both $\sigma$ and $\sigma'$ are irreducible of $\sp_2\fq$, and the descent of them both are the trivial representation of the trivial group $``\sp_0\fq"$. But the descent sequence index $\hat{\ell}(\pi)=(0,2)$ is not a partition. And the wavefront set is corresponding to the partition $(1,1)$ which is equal to $\widetilde{\ell}(\pi)$.
\end{example}

In above example, one can view the trivial representation $\bf{1}$ as a ``non-tempered'' representation in some sense, which explain that the descent method may not directly give us a partition or furthermore the wavefront set.

Since the descent method actually works on the $F$-rational points $G^F$, we will not only get the wavefront set $\wco$ of $\pi$, but also find out that for which $F$-rational nilpotent orbits $\co$ in the $F$-stable nilpotent orbit $\wco$, the representation $\pi$ appears in the generalised Gelfand-Graev representation corresponding to $\co$. The explicit construction of $\co$ will be found in subsection \ref{sec8.4}. In other words, we construct the generalized Whittaker model of $\pi$.

\begin{theorem} \label{main2}
(i) Assume that $\pi$ is in Case (A). Then for any $F$-rational nilpotent orbit $\co\subset \wco_{\hat{\ell}{(\pi)}}$, we have
\[
\langle \pi,\Gamma_\gamma\rangle=1
\]
where $\gamma=\{X,H,Y\}$ with $X\in \co$.

(ii) Assume that $\pi$ is in Case (B) and $\pi$ has an irreducible largest descent index (see section \ref{sec5.4} for detail). Then for any $F$-rational nilpotent orbit $\co\subset \wco_{\hat{\ell}{(\pi)}}$, we have
\[
\langle \pi,\Gamma_\gamma\rangle\le1
\]
where $\gamma=\{X,H,Y\}$ with $X\in \co$.

(iii) Assume that $\pi$ is in Case (C). Then for any $F$-rational nilpotent orbit $\co\subset \wco_{\hat{\ell}{(\pi)}}$, we have
\[
\langle \pi,\Gamma_\gamma\rangle\le1
\]
where $\gamma=\{X,H,Y\}$ with $X\in \co$.
\end{theorem}

\begin{remark}
In above theorem, for each $F$-rational nilpotent orbit $\co$, the multiplicity $\langle \pi,\Gamma_\gamma\rangle$ could be calculated explicitly by our proof of main results. To be more precise, for an irreducible representation $\pi\in\rm{Irr}(\sp_{2n}\fq)$, let $\psi$ be a non-trivial character of $\Fq$ such that $\ell_{0,\psi}^{\rm{FJ}}(\pi)\ge\ell_{0,\psi'}^{\rm{FJ}}(\pi) $ for any character $ \psi'$ of $\Fq$. Let $\pi'$ be an irreducible component of $\CD^\rm{FJ}_{\ell_0, \psi}(\pi)$ (see subsection \ref{sec5.4} for detail). Here we can describe each irreducible component of $\CD^\rm{FJ}_{\ell_0, \psi}(\pi)$ explicitly (see section \ref{sec6}). Let $\co'$ be the $F$-rational nilpotent orbit in the $F$-stable orbit of the wavefront set of $\pi$ such that
$\cal{F}(\pi',\co',\psi' )\ne0$ (see subsection \ref{sec5.3}). According to the classification of the $F$-rational nilpotent orbits of $\sp_{2n}\fq$ in Theorem \ref{ra}, we assume that $\co'$ is corresponding to $(\lambda, (q_i))$. Then
\[
\langle \pi,\Gamma_\gamma\rangle=1
\]
if and only if one can pick an irreducible component $\pi'$ of $\CD^\rm{FJ}_{\ell_0, \psi}(\pi)$ and $\co'$ such that $\co$ is corresponding to
\[
\left\{
\begin{array}{ll}
(\ell_{0,\psi}^{\rm{FJ}}(\pi),\lambda_1,\lambda_2,\cdots),(\epsilon,\epsilon q_1, \epsilon q_2,\cdots)),&\textrm{if }\ell_{0,\psi}^{\rm{FJ}}(\pi)>\lambda_1;\\
(\ell_{0,\psi}^{\rm{FJ}}(\pi),\lambda_1,\lambda_2,\cdots),((-1)^{\#\{i|q_i=\ell_{0,\psi}^{\rm{FJ}}(\pi)\}}\epsilon,\epsilon q_2,\cdots)),&\textrm{otherwise}
\end{array}\right.
\]
where
\[
\epsilon=\left\{
\begin{array}{ll}
+,&\textrm{if $\psi'$ is in the square class of }\psi;\\
-,&\textrm{otherwise}.
\end{array}\right.
\]
Hence, we can calculate the multiplicity $\langle \pi,\Gamma_\gamma\rangle$ by induction on $n$.

For example, by the main theorem in \cite{Wang1} and our Theorem \ref{main2}, for the unipotent cuspidal representation $\pi$ of $\sp_{2k(k+1)}\fq$,
\[
\langle \pi,\Gamma_\gamma\rangle=1
\]
if and only if
$\co$ is corresponding to $((2k,2k,2(k-1),2(k-1),2(k-2),2(k-2),\cdots),(q_i))$ with $q_i=q_{i+1}$ for odd $i$. This example shows that there exists an irreducible representation $\pi$ such that it does not appear in the generalised Gelfand-Graev representation of all of $F$-rational orbits $\co\subset \wco_{\hat{\ell}{(\pi)}}$. Comparing this example to Theorem \ref{main2} (i), and one shall find out how tricky things can be for Case (B). On the other hand, since $\pi$ is cuspidal, this example shows that there exists $F$-rational orbits $\co\subset \wco_{\hat{\ell}{(\pi)}}$ such that $\pi$ does not appear in the generalised Gelfand-Graev representation of $\co$. Therefore the inequality in Theorem \ref{main2} (iii) can not be equality.
\end{remark}

 In \cite{LZ}, B. Liu and Q. Zhang mention that the theorem of multiplicity one of the representations implies the existence of certain twisted
gamma factors as in \cite{Ro}. In \cite{N}, Nien proved a local converse theorem for $\GGL_n\fq$ by using the local
gamma factors and the proof is based on the Whittaker model and Bessel functions.

This paper is organized as follows. In section \ref{sec2}, we describe some preliminaries and the partitions and symbols which will be used to parameterize irreducible representations of $G^F$. In Section \ref{sec3}, we recall the theory of Deligne-Lusztig characters and the Lusztig correspondence which will be used in Section \ref{sec4}. In Section \ref{sec4},  we give a parametrization of irreducible representations of symplectic groups over finite fields. In Section \ref{sec5}, we first review the parametrization of nilpotent orbits in the classical Lie algebras, following the book by Collingwood and McGovern \cite{CM} and the construction of the generalised Gelfand-Graev representations, following by Luszig \cite{L7} and Gomez-Zhu \cite{GZ}. Then we describe the relationship between the generalised Gelfand-Graev representation and the descent method. In Section \ref{sec6}, we review our pervious descent results and calculate the array $\hat{\ell}(\pi)$ and the partition $\widetilde{\ell}(\pi)$ of $\pi\in\rm{Irr}(G^F)$. In Section \ref{sec7}, we review \cite{GRS} and we extend their results to obtain a finite fields version of ``Exchanging roots Lemma'' which are main tools to prove our theorems. In Section \ref{sec8}, we prove our main theorems.

\subsection*{Acknowledgement} We acknowledge generous support provided by National natural Science Foundation of PR China (No. 12071326) and China Postdoctoral Science Foundation (No. 2021TQ0233, No. 2021M702399).
\section{Preliminaries}\label{sec2}
\subsection{Partitions}
Let $\lambda=(\lambda_1,\lambda_2,\cdots,\lambda_k)$ be a partition of $n$, and let $|\lambda|:=\sum_i\lambda_i$. For simplicity of notations, we write $\lambda=(a_1^{k_1},\cdots,a_l^{k_l})$ where $a_j$ runs over $\{\lambda_i\}$ and $k_j=\#\{i|\lambda_i=a_j\}$.
We set
\[
a\lambda:=(a\lambda_1,a\lambda_2,\cdots,a\lambda_k) \ \textrm{ with }a\in\bb{Z}
\]
and
\[
\lambda^{-j}:=(\lambda_1-j,\lambda_2-j,\cdots,\lambda_k-j).
\]
It is obvious that one can get $\lambda^{-j}$ by removing the first $j$ columns of $\lambda$. As is standard, we realize partitions as Young diagrams, and denote by $\lambda^t=(\lambda^t_1,\lambda^t_2,\cdots)$ the transpose of $\lambda$. Then $(\lambda^{-j})^t_1=\lambda^t_{j+1}$.

To each partition $\lambda=(\lambda_1,\cdots,\lambda_k)$, we set
\[
\lambda^* = (\lambda^*_i),\textrm{ defined by }\lambda^*_i:= \lambda_i+(k-i), \textrm{ for }1 \le i \le k.
\]
We then divide $\lambda^*$ into its odd parts $(\lambda^*_{\rm{odd},i})$ and even parts $(\lambda^*_{\rm{even},i})$.
Let
\[
(\eta,\zeta)=(\eta_1,\cdots,\eta_m,\zeta_1,\cdots,\zeta_r)
\]
 be a bi-partition such that
\[
\lambda^*_{\rm{odd},i}=2\eta^*_i+1\textrm{ and }\lambda^*_{\rm{even},i}=2\zeta^*_i
\]
with $\eta^*_i:=\eta_i+(m-i)$ and $\zeta^*_i:=\eta_i+(r-i)$. Thus we obtain a injective map
\begin{equation}\label{p1}
\psi:\lambda\to (\eta,\zeta).
\end{equation}

\subsection{Symbols}
We follow the notation of \cite{P3}. The notation is slightly different from that of \cite{L1}.

A symbol is an array of the form
\[
\Lambda=
\begin{pmatrix}
A\\
B
\end{pmatrix}
=
\begin{pmatrix}
a_1,a_2,\cdots,a_{m_1}\\
b_1,b_2,\cdots,b_{m_2}
\end{pmatrix}
\]
of two finite sets $A$, $B$ of natural numbers (possibly empty) with $a_i, b_i\ge0$, $a_i>a_{i+1}$ and $b_i>b_{i+1}$.

The rank and defect of a symbol $\Lambda$ are defined by
\[
\begin{aligned}
&\rm{rank}(\Lambda)=\sum_{a_i\in A}a_i+\sum_{b_i\in B}b_i-\left\lfloor\left(\frac{|A|+|B|-1}{2}\right)^2\right\rfloor, \\
&\rm{def}(\Lambda)=|A|-|B|
\end{aligned}
\]
where $|X|$ denotes the cardinality of a finite set $X$. Note that the definition of $\rm{def}(\Lambda)$ differs from that of \cite[P.133]{L1}.

For a symbol $\Lambda=\begin{pmatrix}
A\\
B
\end{pmatrix}$, let $\Lambda^*$ (resp. $\Lambda_*$) denote the first row (resp. second row) of $\Lambda$, i.e. $\Lambda^*=A$ and $\Lambda_*=B$. For a symbol $\Lambda=\begin{pmatrix}
A\\
B
\end{pmatrix}$, let
$\Lambda^t=\begin{pmatrix}
B\\
A
\end{pmatrix}$.

Define an equivalence relation generated by the rule
\[
\begin{pmatrix}
a_1,a_2,\cdots,a_{m_1}\\
b_1,b_2,\cdots,b_{m_2}
\end{pmatrix}
\sim
\begin{pmatrix}
a_1+1,a_2+1,\cdots,a_{m_1}+1,0\\
b_1+1,b_2+1,\cdots,b_{m_2}+1,0
\end{pmatrix}.
\]
Note that the defect and rank are functions on the set of equivalence classes of symbols.

A symbol
\[
\cal{Z}=\begin{pmatrix}
a_1,a_2,\cdots,a_{m_1}\\
b_1,b_2,\cdots,b_{m_2}
\end{pmatrix}
\]
with $\rm{def}(\cal{Z}) = 0,1$ is called \emph {special} if
\[
a_1\ge b_1\ge a_2\ge b_2\ge\cdots.
\]
\subsection{Bi-partitions}\label{sec4.2}
For two partitions $\lambda=[\lambda_1,\lambda_2,\cdots,\lambda_k]$ and $\mu=[\mu_1,\mu_2,\cdots,\mu_l]$, we denote
\[
\lambda\preccurlyeq\mu\quad\textrm{if }\mu_i-1\le \lambda_i \le\mu_i\textrm{ for each }i.
\]

Let $\mathcal{P}_2(n)=\left\{\begin{bmatrix}
\lambda\\
\mu
\end{bmatrix}\right\}$
denote the set of bi-partitions of $n$ where $\lambda$, $\mu$ are partitions
and $|\lambda| + |\mu| = n$. To each symbol we can associate a bi-partition as follows:
\[
\Upsilon: \Lambda=\begin{pmatrix}
a_1,a_2,\dots,a_{m_1}\\
b_1,b_2,\dots,b_{m_2}
\end{pmatrix}
\mapsto
\begin{bmatrix}
a_1-(m_1-1),a_2-(m_1-2),\dots,a_{m_1-1}-1,a_{m_1}\\
b_1-(m_2-1),b_2-(m_2-2),\dots,b_{m_2-1}-1,b_{m_2}
\end{bmatrix}=\begin{bmatrix}
\lambda\\
\mu
\end{bmatrix}.
\]
We write
$\Upsilon(\Lambda)^*=\lambda$ and $\Upsilon(\Lambda)_*=\mu.$
Then we have a bijection£º
\begin{equation}\label{bp}
\Upsilon:\mathcal{S}_{n,\beta}\to \left\{
\begin{array}{ll}
\mathcal{P}_2(n-(\frac{\beta+1}{2})(\frac{\beta-1}{2})), &  \textrm{if }\ \beta\textrm{ is odd},\\
\mathcal{P}_2(n-(\frac{\beta}{2})^2), & \textrm{if }\ \beta\textrm{ is even}.
\end{array}\right.
\end{equation}
where $\mathcal{S}_{n,\beta}$ denotes the set of symbols of rank $n$ and defect $\beta$.

Let
$
(\eta,\zeta)=(\eta_1,\cdots,\eta_m,\zeta_1,\cdots,\zeta_r)
$
 be a bi-partition. We can put $r=m-1$ by adding zeros as parts where necessary. We then associate $(\eta,\zeta)$ with a symbol $\Lambda = \Lambda_{\eta,\zeta}$ of defect 1 to be defined by
\[
\begin{pmatrix}
\eta_1+(m-1),\eta_2+(m-2),\dots,\eta_{m-1}+1,\eta_m\\
\zeta_1+(r-1),\zeta_2+(r-2),\dots,\zeta_{r-1}+1,\zeta_{r}
\end{pmatrix}.
\]
It is easy to see that $\Upsilon(\Lambda_{\eta,\zeta})=\begin{bmatrix}\eta\\\zeta\end{bmatrix}$.
\section{Deligne-Lusztig characters and Lusztig correspondence} \label{sec3}
Let $G$ be a connected reductive algebraic
group over $\mathbb{F}_q$. In \cite{DL}, P. Deligne and G. Lusztig defined a virtual character $R^{G}_{T,\theta}$ of $G^F$, associated to an $F$-stable maximal torus $T$ of $G$ and a character $\theta$ of $T^F$. We review some standard facts on these characters and Lusztig correspondence (cf. \cite[Chapter 7, 12]{C}, \cite[section 5]{Gec}), which will be used in this paper.

More generally, let $L$ be an $F$-stable Levi subgroup of a parabolic subgroup $P$ which is not necessarily $F$-stable, and $\pi$ be a representation of the group $L^F$. Then $R^G_L(\pi)$ is a virtual character of $G^F$. If $P$ is $F$-stable, then the Deligne-Lusztig induction coincides with the parabolic induction
\[
R^G_L(\pi)= \rm{Ind}^{G^F}_{P^F}(\pi).
\]
For example if $T$ is contained in an $F$-stable Borel subgroup $B$, then
\[
R^G_{T,\theta}=\rm{Ind}^{G^F}_{B^F}\theta.
\]

\subsection{Centralizer of a semisimple element }\label{sec2.2}

 Let $G$ be a classical group over finite field and $F$ a Frobenius morphism of $G$. Let $s$ be a semisimple element in the connected component of $G$. Let $C_{G(\overline{\bb{F}}_q)}(s)$ be the centralizer in $G(\overline{\bb{F}}_q)$ of a semisimple element $s \in G^0(\overline{\bb{F}}_q)$. In \cite[subsection 1.B]{AMR}, A.-M. Aubert, J. Michel and R. Rouquier described $C_{G(\overline{\bb{F}}_q)}(s)$ as follows. Let $T(\overline{\bb{F}}_q) \cong \overline{\bb{F}}^\times_q\times\cdots\times\overline{\bb{F}}_q^\times$ be a $F$-rational maximal torus of $G(\overline{\bb{F}}_q)$, and let $s=(x_1,\cdots,x_l)\in T^F$. Let $\nu_{a}(s):=\#\{i|x_i=a\}$ and let
  \[
  [ a] := \{a^{{q}^k}|k\in\bb{Z}\}.
  \]
  Clearly, if $a'\in[a]$ and $a\in\{x_i\}$, then $a'\in\{x_i\}$ and $\nu_{a'}(s)=\nu_{a}(s)$.
 The group $C_{G(\overline{\bb{F}}_q)}(s)$ has a natural decomposition with the eigenvalues of $s$:
\[
C_{G(\overline{\bb{F}}_q)}(s)=\prod_{[ a] \subset \{x_i\}}G_{[ a]}(s)(\overline{\bb{F}}_q)
\]
where $G_{[ a]}(s)(\overline{\bb{F}}_q)$ is a reductive quasi-simple group of rank equal to $\#[ a]\cdot\nu_{a}(s)$. In particular, for special orthogonal group $G=\so_{2n+1}$ which is the dual group of symplectic group $\sp_{2n}$, we have
\begin{itemize}

\item If $a=1$, then $G_{[ 1]}(s)=\so_{2\nu_{1}(s)+1}$;

\item If $a=-1$ and $G_{[ -1]}(s)=\o^\epsilon_{2\nu_{-1}(s)}$;

\item  If $a\ne \pm 1$, then $G_{[ a]}(s)(\overline{\bb{F}}_q)$ is either $\GGL_{\nu_a(s)}$ or $\UU_{\nu_a(s)}$.

\end{itemize}

For $[a]\nsubseteq \{x_i\}$, we set $G_{[ a]}(s):=1$. Then we rewrite
\begin{equation}\label{decomp}
C_{G^F}(s)=\prod_{[ a] }G^F_{[ a]}(s)
\end{equation}
where the product runs over every $[a]$ with $a\in \overline{\bb{F}}_q$.

\subsection{Lusztig correspondence for symplectic groups}\label{mlus}
In this subsection we shall recall the Lusztig correspondence and Jordan decomposition. The content can be found in \cite{DL,L1}, or in a survey given by Geck \cite[section 5]{Gec}. Let $G$ be a symplectic group defined over $\Fq$, and $G^*$ be the dual group of $G$. We still denote the Frobenius endomorphism of $G^*$ by $F$. Then there is a natural bijection between the set of $G^F$-conjugacy classes of $(T, \theta)$ and the set of $G^{*F}$-conjugacy classes of $(T^*, s)$ where $T^*$ is a $F$-stable maximal torus in $G^*$ and $s \in   T^{*F}$. We will also denote $R_{T,\theta}^G$  by $R_{T^*,s}^G$ if $(T, \theta)$ corresponds to $(T^*, s)$.
For a semisimple element $s \in G^{*F}$, define
\[
\mathcal{E}(G^F,s) = \{ \chi \in \rm{Irr}(G^F)  :  \langle \chi, R_{T^*,s}^G\rangle \ne 0\textrm{ for some }T^*\textrm{ containing }s \}.
\]
The set $\mathcal{E}(G,s)$ is called  the Lusztig series. We can thus define a partition of $\rm{Irr}(G^F)$ by Lusztig series (see e.g. \cite[subsection 2.3]{Ma} and \cite[Theorem 5.2]{Gec})
i.e.
\[
\rm{Irr}(G^F)=\coprod_{(s)}\mathcal{E}(G^F,s).
\]

\begin{proposition}[Lusztig's Jordan decomposition](see e.g. \cite[Theorem 8.14]{Sr2})\label{Lus}
Let $I$ be the identity in the dual group of $C_{G^{*F}}(s)$. There is a bijection
\[
\mathcal{L}_s:\mathcal{E}(G^F,s)\to \mathcal{E}(C_{G^{*F}}(s),I),
\]
extended by linearity to a map between virtual characters satisfying that
\begin{equation}\label{Lus2}
\mathcal{L}_s(\varepsilon_G R^G_{T^*,s})=\varepsilon_{C_{G^{*}}(s)} R^{C_{G^{*F}}(s)}_{T^*,1}.
\end{equation}
Moreover, we have
\[
\rm{dim}(\pi)=\frac{|G^F|_{p'}}{|C_{G^{*F}}(s)|_{p'}}\rm{dim}(\cal{L}_s(\pi))
\]
where $|G|_{p'}$ denotes greatest factor of $|G|$ not divided by $p$ with $\Fq=\bb{F}_{p^m}$, and  $\varepsilon_G:= (-1)^r$ where $r$ is the $\Fq$-rank of $G$.
Moreover, Lusztig correspondence sends cuspidal representation to cuspidal representation (see e.g. \cite[Chap 9]{L1} and \cite[Lemma 2.7]{Ma}).
\end{proposition}

Note that the correspondence $\cal{L}_s$ is usually not uniquely determined. We say a complex irreducible representation is uniform if its character is a linear combination of the Deligne-Lusztig characters. In this paper, we only
use Lusztig correspondence in the uniform case, which is uniquely determined by (\ref{Lus2}). By (\ref{decomp}), we can rewrite Lusztig correspondence as
\begin{equation}\label{lc}
\begin{matrix}
\mathcal{L}_s:&\mathcal{E}(G^F,s)&\to &\prod_{[ a] }\mathcal{E}(G^{*F}_{[ a]}(s),I)&\\
\\
&\pi&\to&\prod_{[a]}\pi[a].
\end{matrix}
\end{equation}

\section{Classification of irreducible representations for finite classical groups}\label{sec4}

In this section, we first review some results on the classification of the irreducible unipotent representations of classical groups by Lusztig in \cite{L1, L3, L4, LS}. Then we give a parametrization of irreducible representations of symplectic groups.
\subsection{Classification of unipotent representations}

An irreducible representation $\pi$ of $G^F$  is called unipotent if
\[
\pi\in \cal{E}(G^F,I).
\]

The classification of the representations of $G^F=\GGL_n\fq$ and $\UU_n(\Fq)$ was given by Lusztig and Srinivasan in \cite{LS}. Denote by $W_n\cong S_n$ the Weyl group of the diagonal torus $T_0$ in  $\GGL_n\fq$ or $\UU_n(\Fq)$. For any $F$-stable maximal torus $T$, there is $g\in G$ such that $^gT=T_0$. Since $T$ is $F$-stable, we have $gF(g^{-1})\in N_G(T_0)$. If $w$ is the image of  $gF(g^{-1})$ in $W_n$, then we denote $T$ by $T_w$.

\begin{theorem}\label{thm3.3}
Let $G_n:=\GGL_n$ (resp. $\UU_n$). Let $\sigma$ be an irreducible representation of $S_n$. Then
\[
R_\sigma^{G_n}:=\frac{1}{|W_n|}\sum_{w\in W_n}\sigma(ww_0)R_{T_w,1}^{G_n}
\]
is (resp. up to sign) a unipotent representation of $\GGL_n\fq$ (resp. $\UU_n(\Fq)$) and all unipotent representations of $\GGL_n\fq$ (resp. $\UU_n(\Fq)$) arise in this way.\end{theorem}

It is well-known that irreducible representations of $S_n$ are parametrized by partitions of $n$. For a partition $\lambda$ of $n$,  denote by $\sigma_\lambda$ the corresponding representation of $S_n$, and
let $\pi_\lambda= R_{\sigma_\lambda}^{\GGL_n}$ (resp. $\pm R_{\sigma_\lambda}^{\UU_n}$) be the corresponding unipotent representation of $\UU_n(\Fq)$.
By Lusztig's result \cite{LS},  $\pi_\lambda$ is (up to sign) a unipotent cuspidal representation  of $\UU_n(\Fq)$ if and only if $n=\frac{k(k+1)}{2}$ for some positive integer $k$ and  $\lambda=[k,k-1,\cdots,1]$.

Now we recall the parametrization of irreducible unipotent representations of symplectic groups and orthogonal groups. Lusztig gives a bijection between the unipotent representations of these groups to equivalence classes of symbols as follow:
\[
\left\{
\begin{aligned}
&\cal{E}(\sp_{2n},I)\\
&\cal{E}(\o_{2n+1},I)\\
&\cal{E}(\o^+_{2n},I)\\
&\cal{E}(\o^-_{2n},I)
\end{aligned}\right.
\longrightarrow
\left\{
\begin{aligned}
&\cal{S}_n:=\big\{\Lambda|\rm{rank}(\Lambda)=n, \rm{def}(\Lambda)=1\ (\textrm{mod }4)\big\};\\
&\cal{S}_n\times\{\pm\};\\
&\cal{S}^+_n:=\big\{\Lambda|\rm{rank}(\Lambda)=n, \rm{def}(\Lambda)=0\ (\textrm{mod }4)\big\};\\
&\cal{S}^-_n:=\big\{\Lambda|\rm{rank}(\Lambda)=n, \rm{def}(\Lambda)=2\ (\textrm{mod }4)\big\};
\end{aligned}\right.
\]
Let $\pi_{\Lambda}$ be the irreducible representation parametrized by $\Lambda$.

\subsection{Classification of irreducible representations of symplectic groups}\label{4.4}
\begin{definition}
For $G^F =\sp_{2n}\fq$ we have $G^{*F} = \so_{2n+1}\fq$. Let $s = (-I, 1)\in \so_{2n+1}\fq$ with $I$ being the identity in $\so_{2n}^\epsilon\fq $. We say that $\pi$ is a $\theta$-epresentation if $\pi\in\cal{E}(G,s)$.

\end{definition}

Let $G$ be a symplectic group. By the Lusztig correspondence, there is a bijection
\[
\mathcal{L}_s:\mathcal{E}(G^F,s)\to \prod_{[ a] }\mathcal{E}(G^{*F}_{[ a]}(s),1)=\mathcal{E}(G^{*F}_{[ \ne\pm1]}(s),1)\times\mathcal{E}(G^{*F}_{[ 1]}(s),1)\times\mathcal{E}(G^{*F}_{[ -1]}(s),1).
\]
where
\begin{itemize}
\item $G^{*F}_{[ \ne\pm1]}(s)=\prod_{[ a]\ne[\pm1] }G^{*F}_{[ a]}(s)$;

\item $G^{*F}_{[ a]}(s)$ is either a general linear group or unitary group if $a\ne\pm1$;

\item $G^*_{[ 1]}(s)=\so_{2\nu_{1}(s)+1}$;

\item $G^*_{[ -1]}(s)=\o^\epsilon_{2\nu_{-1}(s)}$;

\end{itemize}
(c.f. Section \ref{sec2.2} for details).

Based on above bijection and Lusztig's classification of unipotent representations \cite[Chap VIII]{Sr2},\cite{L3}, we obtained in \cite[subsection 4.4]{Wang1} the following classification of irreducible representations:
\[
\mathcal{E}(G^F,s)\longrightarrow\mathcal{E}(G^{*F}_{[ \ne\pm1]}(s),I)\times\cal{S}_{\nu_1(s)}\times\cal{S}^\epsilon_{\nu_{-1}(s)}.
\]

Recall that the Lusztig correspondence is not uniquely determined. The parametrization of irreducible representations depends on the choice of the Lusztig correspondence.
Let
\begin{equation}\label{lus3}
\cal{L}_G: \cal{E}(G)=\coprod_{(s)}\mathcal{E}(G^F,s)\to\coprod_{(s)}
\mathcal{E}(G^{*F}_{[ \ne\pm1]}(s),I)\times\mathcal{E}(G^{*F}_{[ 1]}(s),I)\times\mathcal{E}(G^{*F}_{[ -1]}(s),I)
\end{equation}
such that for $\pi\in\cal{E}(G,s)$, we have
\[
\cal{L}_G(\pi)=
\cal{L}_s(\pi).
\]
We call $\cal{L}_G$ the Lusztig correspondence for $G$. For a fixed $\cal{L}_G$, let $\pi_{\rho,\Lambda,\Lambda'}$ denote the irreducible representation
parametrized by the pair of symbols $(\Lambda,\Lambda')$ and an irreducible unipotent representation $\rho$ of $G^F_{[ \ne\pm1]}(s)$.
From now on, we fix a choice of a series Lusztig correspondences $\{\cal{L}_{\sp_{2n}}\}$ of $\{\sp_{2n}\fq\}$ satisfying the conditions in \cite[Subsection 4.4]{Wang1}. Thus we fix a parametrization for irreducible representations of symplectic groups.

Note that $\pi_{-,\Lambda,-} $ is a unipotent representation of symplectic group and $\pi_{-,\Lambda,-}=\pl$ where we write blank by $-$. On the other hand, $\pi_{-,-,\Lambda'} $ is a $\theta$-representation of symplectic group. We still denote $\pi_{-,-,\Lambda'} $ by $\pi_{\Lambda'}$. In these notations, one should not confuse unipotent representations with $\theta$-representations, because $\pl$ is unipotent (resp. $\theta$) if and only if $\rm{def}(\Lambda)$ is odd (resp. even).

Assume that $\pi\in\cal{E}(G^F,s)$ and $s$ has no eigenvalues $\pm1$, i.e. we assume that $\pi=\pi_{\rho,-,-}$. We set
\begin{equation}\label{a}
\cal{L}_s(\pi_{\rho,-,-})=\rho=\prod_{[a]\ne[\pm1]}\pi{[a]}
\end{equation}
where $\pi[a]$ are unipotent representations of $G^F_{[a]}(s)$. Recall that $G_{[a]}(s)$ is either a general linear group or a unitary group. Then by the classification of unipotent representations of these groups, we can associate each $\pi{[a]}$ with a partition $\lambda{[a]}$.

\section{Nilpotent Orbits and generalized Gelfand-Graev representations}\label{sec5}
Let $G$ be a connected reductive algebraic group defined over $\Fq$ and $\fg$ its Lie algebra, on which we
fix an Ad $G$-invariant non-degenerate bilinear form $\kappa$.

\subsection{$\frak{sl}_2$-triples}
The standard references for the classification of nilpotent orbits can be found in
\cite{CM} or \cite[Chapter 5]{C}.
We will review the basic results on nilpotent orbits and $\frak{sl}_2$-triples in $\fg$.

A $\frak{sl}_2$-triple is a Lie algebra homomorphism $\gamma:\frak{sl}_2\to \fg$.
Let $\Theta$ be the set of all $\frak{sl}_2$-triples and for $\gamma\in \Theta$, we set
\[
X:=\gamma\begin{pmatrix}0&1\\
0&0
\end{pmatrix},\
Y:=\gamma\begin{pmatrix}0& 0\\
1& 0
\end{pmatrix}
,\
H:=\gamma\begin{pmatrix}1& 0\\
0& -1
\end{pmatrix}.
\]
We regard $\Theta$ as a closed subvariety of the vector space $\fg^{\otimes 3}$, via $\gamma\to \{X,Y,H\}$.

For a fixed $\frak{sl}_2$-triple,
let $\fg_i = \{Z \in \fg | \rm{ad}(H)(Z) = iZ\}$ with $i \in \bb{Z}$. Then, from standard $\frak{sl}_2$-theory, we
have the decomposition
\[
\fg = \bigoplus_{j\in\bb{Z}}\fg_{j}.
\]
Let
\[
\fg_{\ge i}:=\bigoplus_{i\ge j}\fg_{j}\textrm{ and } \fg_{\le i}:=\bigoplus_{i\le j}\fg_{j}.
\]
They are Lie algebras of a close connect subgroup $G_{\ge i}$ and $G_{\le i}$ of $G$, respectively.

From the well-known results of Jacobson-Morozov and Kostant \cite[Chapter
3]{CM}, there is a 1-1 correspondence
\[
\begin{matrix}
\left\{\begin{matrix}
\textrm{Ad}(G)\textrm{ conjugacy classes of}\\
 \textrm{ $\frak{sl}_2$-triples in $\fg$}
\end{matrix}
\right\}
&\longleftrightarrow&
\left\{\begin{matrix}
\textrm{Nonzero nilpotent Ad}(G)\textrm{-orbits}\\
\wco\subset\fg
\end{matrix}
\right\}\\
\\
\gamma = \{H, X, Y \}&\longleftrightarrow&\wco=\rm{Ad}(G)\cdot X
\end{matrix}.
\]
If the conjugacy class of an $\frak{sl}_2$-triple $\gamma$ corresponds to a nilpotent orbit $\wco\subset\fg$, then
we say that $\gamma$ is an $\frak{sl}_2$-triple of type $\wco$. We also call these $\rm{Ad}(G)$-orbits $\wco$ the $F$-stable nilpotent orbits in $\fg$.
\subsection{$F$-rational orbits of symplectic groups}
In this subsection, we will classify $F$-rational orbits $\co$ for each $F$-stable nilpotent orbit $\wco$.

From now on, assume that $G$ is a symplectic group and $G^F=\sp_{2n}\fq$ naturally acts on $V$ where $V$ is an $\Fq$-vector space endowed with a symplectic form $(,)$.
Recall that the set of $F$-stable nilpotent orbits of symplectic groups are parameterized by partitions where odd numbers occur with even multiplicity. Hence each this kind of partition $\lambda$
defines a $F$-stable nilpotent orbit $\wco_\lambda$. More explicitly, one can attach a $F$-stable nilpotent orbit $\wco$ to a partition as follow.

Since $G$ (hence $\fg$) is defined over $\Fq$, the subvariety $\Theta$ of $\fg^{\otimes 3}$ is also naturally defined over $\Fq$. And there exists a maximal torus of symplectic group $G$ which is defined  and split over $\Fq$. Then for any $F$-stable nilpotent orbit $\wco$, we have $F(\wco)= \wco$.
Fixed a $F$-stable nilpotent orbit $\wco$, pick a $\frak{sl}_2$-triple $\gamma=\{X,Y,H\}\in \Theta^F$ such that $X\in\wco$.
Set $\fg_{\gamma} = \rm{Span}\{X, H, Y \} \subset \fg^F$. The $\frak{sl}_2$-triple $\gamma$ give rise to a decomposition of $V$:
\[
V=\bigoplus_{i}n_iV_i
\]
where $V_i$ is a irreducible $i$-dimensional $\fg_{\gamma}$-modules and  $n_i$ is the multiplicity of $V_i$ in the decomposition of $V$ by $\fg_{\gamma}$. Then we get the partition corresponding to $\wco$:
\[
\wco=\rm{Ad}(G)\cdot X\longleftrightarrow \gamma = \{H, X, Y \} \longrightarrow (1^{n_1},2^{n_2},\cdots).
\]
We denote the above $F$-stable nilpotent orbit $\wco$ by $\wco_\lambda$.

Methods of \cite[I.6]{Wal} imply the following parameterization of $F$-rational orbit of nilpotent elements in $\fg^F$. It also can be found in \cite[section 3]{GZ}.
\begin{theorem}\label{ra}
Let $G=\sp_{2n}$ and $\cal{P}(n)$ be the set of partitions of 2n where
odd numbers occur with even multiplicity.
There is a 1-1 correspondence between the following sets:
\[
\begin{aligned}
\left\{\begin{matrix}
F\textrm{-rational orbit of nilpotent elements in }\fg^F
\end{matrix}
\right\}
\longleftrightarrow
\left\{\Gamma=(\lambda,(q_i))\left|\begin{aligned}\textrm{ $\lambda=(a_1^{k_1},\cdots,a_r^{k_r})\in\cal{P}(n)$, and}\\
\textrm{for any even $a_i$, set $q_i=Q(k_i)$,}\\
 \textrm{and there is no $q_i$ for odd $a_i$; }
\end{aligned}\right.
\right\}
\end{aligned}
\]
where $Q(k_i)$ is a equivalence class of non-degenerate quadratic form of dimension $k_i$.
\end{theorem}

 We denote the $F$-rational nilpotent orbit corresponding to $\Gamma$ by $\co_\Gamma$. If $\Gamma=(\lambda,(q_i))$, then $\co_\Gamma\subset \wco_{\lambda}$ and we say the orbit $\co_\Gamma$ is corresponding to partition $\lambda$. It is well known that for a fixed vector space $V$ with $\rm{dim}(V)$ even over finite fields $\Fq$ with odd $q$, there are precisely two equivalence classes of non-degenerate quadratic forms $Q_1(k)$ and $Q_2(k)$:
 \[
 Q_1(k)=a\cdot Q_2(k)\textrm{ with }a\in \Fq^\times/(\Fq^\times)^2.
 \]
If $\rm{dim}(V)$ is odd, there is only one equivalence class of non-degenerate quadratic form.
\begin{example}
Let $G=\sp_2$. There is exactly one $F$-rational nilpotent orbit corresponding to partition $(1,1)$: $\left\{\begin{pmatrix}
0&0\\
0&0
\end{pmatrix}\right\}$. There are two $F$-rational nilpotent orbits corresponding to partition $(2)$: $\left\{F\textrm{-rational orbit of }\begin{pmatrix}
0&1\\
0&0
\end{pmatrix}\right\}$ and $\left\{F\textrm{-rational orbit of }\begin{pmatrix}
0&a\\
0&0
\end{pmatrix}\right\}$ with $a\in \Fq^\times/(\Fq^\times)^2$.
\end{example}

 Let $V$ be a vector space with dimension $k$. We classify equivalence classes of non-degenerate quadratic form $Q(k)$ as follow. Let $\rm{det}\ Q(k)$ be the determinant of $Q(k)$ of $V$. We normalize the discriminant by
\[
\rm{disc} \ Q(k)=(-1)^{k(k-1)/2}\det Q(k) \in \Fq^\times/ (\Fq^\times)^2,
\]
such that when $\dim V$ is even, $\rm{disc} \ Q(k)=+1$ if and only if the special orthogonal group $\rm{SO}(V)$ which preserves $Q(k)$ is split. Pick up an anisotropic vector $v_0\in V$, and let $W$ be the orthogonal complement of $v_0$. Then we have
\begin{equation}\label{sgn}
\rm{disc} \ Q(k)= (-1)^{n-1}\cdot (Q(k))(v_0) \cdot \rm{disc} \ (Q(k)|_W)
\end{equation}
where both sides are regarded as square classes in $\Fq^\times/ (\Fq^\times)^2$.
This setting is consistent with \cite[(1.3)]{Wang1}. To simplify notations, we write $q_i=+$ (resp. $q_i=-$) if $\rm{disc} \ q_i=\rm{disc} \ Q_i(k)=+1$ (resp. $\rm{disc} \ q_i=0$).

\subsection{Generalized
Gelfand-Graev representations}\label{sec5.3}
We fix an $\rm{Ad}(G)$-invariant non-degenerate bilinear form $\kappa$ on $\fg$. Let $\psi$ be a non-trivial character of $\Fq$.

Pick a $\frak{sl}_2$-triple $\gamma=\{X,Y,H\}\in \Theta^F$. Then $G_{\ge i}$ and $\fg_{\ge i}$ are defined over $\Fq$.
 We regard $\kappa$ as a linear form of $\fg_{\ge 1}$ defined by
\[
\kappa(x)  =  \kappa(Y, x).
\]
By setting $\langle x,  y\rangle  =  \kappa ([x, y])$, we defines a non-singular symplectic form on $\fg_1$.
We choose a Lagrangian subspace $L$, for $\langle  , \rangle$ and we denote $\fg_{\ge1.5}:=L\oplus\bigoplus_{\ge2}\fg_i$.  Then $\fg_{\ge1.5}$ is the Lie algebra of a closed, connected unipotent subgroup $G_{\ge1.5}$ of $G_{\ge1}$. The restriction of $\kappa$ to $\fg_{\ge1.5}$ vanishes on all commutators of this Lie algebra; hence, the composition
\[
G_{\ge1.5}^F\xrightarrow{\rm{log}}\fg_{\ge1.5}^F\xrightarrow{\kappa}\Fq\xrightarrow{\psi}\bb{C}^\times
\]
is a homomorphism of algebraic groups. By abuse of notations, we still denote above character of $G_{\ge1.5}^F$ by $\psi$.

We now associate to $\gamma$ a generalized Gelfand-Graev representation
\[
\Gamma_\gamma:=\rm{Ind}^{G^F}_{G_{\ge 1.5}^F}\psi
\]
and
\[
q^{\rm{dim}\frac{\fg_1}{2}}\cdot\Gamma_\gamma=\rm{Ind}^{G^F}_{G_{\ge 2}^F}\psi.
\]

Note that $H_\gamma:=G_{\ge 1}^F/[G_{\ge 2}^F,G_{\ge 2}^F]\cong L^F\oplus L^{\vee F} \oplus Z^F$ is a Heisenberg group where $Z = G_{\ge 2}^F/[G_{\ge 2}^F,G_{\ge 2}^F]$ and $L$ Lagrangian subspace of $\fg_1$ as before.
 Then, according to the Stone-von Neumann
theorem, there exists a unique, up to equivalence, irreducible representation $\rho_{\psi}$ of $H_\gamma$ such that the center $Z^F$ of $H_\gamma$ acts by the character $\psi$. By composition with the natural homomorphism from $G_{\ge 1}^F$ to $H_\gamma$, we regard $\rho_{\psi}$ as a representation of $G_{\ge 1}^F$, and according to the uniqueness of $\rho_{\psi}$, we have
\[
\rho_{\psi}=\rm{Ind}^{G_{\ge 1}^F}_{G_{\ge 1.5}^F}\psi.
\]
Since
\[
M^F_X = G^F_\gamma := \{g \in G^F |\rm{Ad}(g)X = X, \rm{Ad}(g)H = H, \rm{Ad}(g)Y = Y \}.
\]
 preserves $\gamma$, it is well-known that there exists a representation of a semi-direct product
$M_X^F\rtimes G_{\ge 1}^F$ which extends the representation $\rho_{\psi }$ of $G_{\ge 1}^F$. We refer to the representation $\omega_{\psi }$ of $M^F_X\rtimes G^F_{\ge 1}$ as the Weil representation associated to $\psi$. Then for a representation $\pi$ of $G^F$, the Hom space
\[
\rm{Hom}_{G_{\ge 1}^F}(\pi,\omega_{\psi })
\]
is naturally a representation of $M_X^F$.
By Frobenius reciprocity,
\[
\rm{Hom}_{G_{\ge 1}^F}(\pi,\omega_{\psi })\cong\rm{Hom}_{G^F}(\pi,\rm{Ind}^{G^F}_{G_{\ge 1.5}^F}\psi)\cong \rm{Hom}_{G_{\ge 1.5}^F}(\pi,\psi).
\]
If the $F$-rational orbit of $X$ is corresponding to $\Gamma$ as in Theorem \ref{ra}, then we denote this representation of $M_X^F$ by $\cal{F}(\pi,\co_\Gamma,\psi )$. We say $\pi$ has a nontrivial Fourier coefficient on $\co_\Gamma$ (resp. $\wco_\lambda$) or $\co_\Gamma$ (resp. $\wco_\lambda$) supports $\pi$ if  $\cal{F}(\pi,\co_\Gamma,\psi )\ne0$ and $\co_\Gamma\subset\wco_\lambda$.

If $\Gamma=((2\ell,1^{2(n-\ell)}),(\epsilon))$, then $M_X\cong\sp_{2(n-\ell)}$ and $\cal{F}(\pi,\co_\Gamma,\psi )$ is a representation of $\sp_{2(n-\ell)}\fq$. Let $G'=\sp_{2(n-\ell)}$ and $\co_{\Gamma'}$ be an $F$-rational nilpotent orbit of $G'$. We say $\co_\Gamma\circ\co_{\Gamma'}$ supports if
\[
\cal{F}(\cal{F}(\pi,\co_\Gamma,\psi ),\co_{\Gamma'},\psi)\ne 0.
\]

\subsection{Gan-Gross-Prasad problem and descents for finite symplectic groups}\label{sec5.4}

The Gan-Gross-Prasad problem \cite{GP1,GP2,GGP1,GGP2} concerns a restriction or branching problem in the representation theory of real or p-adic Lie groups. It also has a global counterpart which is concerned with a family of period integrals of automorphic forms. On the other hand, D.
Jiang and L. Zhang \cite{JZ1} study the local descents for p-adic orthogonal groups, whose results can
be viewed as a refinement of the local Gan-Gross-Prasad conjecture, and the descent method has
important applications towards the global problem (see \cite{JZ2}).
In previous works \cite{LW2,LW3}, we have studied the Gan-Gross-Prasad problem and descent method for finite classical groups.

Roughly speaking, the Gan-Gross-Prasad problem concerns the Fourier coefficient of an irreducible representation $\pi\in\rm{Irr}(G^F)$ on an $F$-rational nilpotent orbits $\co$ in an $F$-stable nilpotent orbit $\wco_\lambda$ with $\lambda=(2\ell,1,1,\cdots,1)$.

Let $P_\ell=M_\ell N_\ell$ be the parabolic subgroup of $\rm{Sp}_{2n}$ with Levi factor $M_\ell\cong \GGL_1^{\ell}\times \sp_{2(n-\ell)}$ and its unipotent radical $N_\ell$ can be written in the form
\[
N_{\ell}=\left\{n=
\begin{pmatrix}
z & y & x\\
0 & I_{n-2\ell} & y'\\
0 & 0 & z^*
\end{pmatrix}
: z\in U_{\GGL_\ell}
\right\},
\]
where the superscript ${}^*$ denotes the  transpose inverse, and $U_{\GGL_\ell}$ is the subgroup of unipotent upper triangular matrices of $\GGL_\ell$.

Consider the $F$-rational nilpotent orbits $\co$ corresponding to $\Gamma=(\lambda,(q_1))=((2\ell,1,1,\cdots,1),(+))$. In this setting, $G_{\ge 1}= N_\ell$ and $M_X=\sp_{2(n-\ell)}$. Let
\[
H:=\sp_{2(n-\ell)}\ltimes N_{\ell}=M_X\ltimes G_{\ge 1}
\]
and
 \[
m_{\psi}(\pi,\pi'):=\rm{dim}\rm{Hom}_{H(\Fq)}(\pi, \pi'\otimes \omega_{\psi})
 \]
with $\pi\in\rm{Irr}(\sp_{2n}\fq)$ and $\pi'\in\rm{Irr}(\sp_{2(n-\ell)}\fq)$.

Note that $m_{\psi}(\pi, \pi')$ depends on ${\psi}$, and that
\[
\rm{Hom}_{H(\Fq)}(\pi\otimes \overline{\omega_{\psi}}, \pi')\cong \rm{Hom}_{\rm{Sp}_{2n-2\ell}(\Fq)}(\CJ'_\ell(\pi\otimes\overline{\omega_{\psi}}), \pi')\cong \rm{Hom}_{M_X^F}(\cal{F}(\pi,\co,\psi), \pi'),
\]
where $\CJ'_\ell(\pi\otimes\overline{\omega_{\psi}})$ is the twisted Jacquet module of $\pi\otimes\overline{\omega_{\psi}}$ with respect to $(N_{\ell}(\Fq), \psi)=(G_{\ge 1}^F,\psi)$. Define the $\ell$-th {\sl Fourier-Jacobi quotient} of $\pi$ with respect to $\psi$ to be
\begin{equation}\label{lfd}
\CQ_{\ell,\psi}^\rm{FJ}(\pi):=\CJ'_\ell(\pi\otimes\overline{\omega_{\psi}})=\cal{F}(\pi,\co_{((2\ell,1^{2(n-\ell)}),(+))},\psi),
\end{equation}
viewed as a representation of $\sp_{2(n-\ell)}(\Fq)$. Define the {\sl first occurrence index} $\ell_0:=\ell_{0,\psi}^\rm{FJ}(\pi)$ of $\pi$ to be the largest nonnegative integer $\ell_0\leq n$ such that $\CQ^\rm{FJ}_{\ell_0,\psi}(\pi)\neq 0$ for some choice of $\psi$. The $\ell_0$-th Fourier-Jacobi quotient of $\pi$ with respect to this chosen $\psi$ is called the {\sl first Fourier-Jacobi descent} of $\pi$ or simply the {\sl descent} of $\pi$, denoted by
\begin{equation}\label{bd}
\CD^\rm{FJ}_{\ell_0, \psi}(\pi) :=\CQ^\rm{FJ}_{\ell_0,\psi}(\pi) \ (\textrm{or simply }\CD^\rm{FJ}_{\psi}(\pi)).
\end{equation}

Let $\psi'$ be another nontrivial additive character of $\Fq$ not in the square class of $\psi$. Note that
\begin{equation}\label{5.2}
\co_{((2\ell,1^{2(n-\ell)}),(+))}=a\cdot\co_{((2\ell,1^{2(n-\ell)}),(-))}
\end{equation}
 with $a\in \Fq^\times/(\Fq^\times)^2$. Wa have
\begin{equation}\label{5.3}
\CD^\rm{FJ}_{\ell_0, \psi'}(\pi)=\cal{F}(\pi,\co_{((2\ell_0,1^{2(n-\ell)}),(-))},\psi)
\end{equation}
Hence, to calculate the wavefront set of $\pi$, one need consider both of $\CD^\rm{FJ}_{\ell_0, \psi}(\pi)$ and $\CD^\rm{FJ}_{\ell_0, \psi'}(\pi)$.

We next turn to consider ``the descent of the descent of $\pi$''. However, the descent $\CD^\rm{FJ}_{\ell_0, \psi}(\pi)$ may not be irreducible. So we shall consider the descent of a irreducible component of $\CD^\rm{FJ}_{\ell_0, \psi}(\pi)$. We call a series of irreducible representations $\{\pi_i\}$ a descent sequence of $\pi$ if
\[
\pi_1\xrightarrow{\ell_1}\pi_2\xrightarrow{\ell_2}\pi_{3}\xrightarrow{\ell_3}\cdots\xrightarrow{\ell_k}\bf{1}
\]
where $\pi_1=\pi$ and $\ell_i=\ell_{0,\psi_{i}}^\rm{FJ}(\pi_{i})$, and $\pi_i$ appears in $\CD^\rm{FJ}_{\ell_{i-1}, \psi_{i-1}}(\pi_{i-1})$ for some choice of $\psi_{i-1}$, and the last $\bf{1}$ is the trivial representation of trivial group. We call the array $(2\ell_1,2\ell_2,\cdots)$ the descent sequence index of $\pi$ with respect to $\{\pi_i\}$. Note that $(2\ell_1,2\ell_2,\cdots)$ may {\sl not} be a partition. For two descent sequence indexes $(2\ell_1,2\ell_2,\cdots)$ and $(2\ell_1',2\ell_2',\cdots)$ of $\pi$, we say $(2\ell_i)\succeq(2\ell_i')$ if for some $j$ we have
\[
\ell_i=\ell'_i\ \textrm{ with }i<j\textrm{ and }\ell_j>\ell'_j.
\]
It easy to check that $\succeq$ is a order of descent sequence indexes, and let $\hat{\ell}{(\pi)}$ or simply $\hat{\ell}{}$ be the largest descent sequence index, and we say these kinds of sequence are the largest descent sequences of $\pi$. For a largest descent sequences of $\pi$, we say the sequence is irreducible if $\CD^\rm{FJ}_{\ell_{i-1}, \psi_{i-1}}(\pi_{i-1})=\pi_i\in\rm{Irr}(\sp_{2n_i}\fq)$ for every $i$, i.e. each step in this largest descent sequence is irreducible.
\section{Calculation of descent sequence index}\label{sec6}
The aim of this section is to recall the descents result for irreducible representations of finite symplectic groups in \cite{Wang1,Wang2} and calculate their (largest) descent sequence. Let $\pi\in\cal{E}(\sp_{2n}\fq,s)$. In this paper, we consider the following three cases:

\begin{itemize}

\item (A) $s$ has no eigenvalues $\pm1$;

\item (B) $s$ only has eigenvalues $1$ or only has eigenvalues $-1$;

\item (C) $\pi$ is cuspidal.
\end{itemize}
We can not get the wavefront sets of the general case directly by composing the first two cases. For example, Let $P$ be a parabolic subgroup of $\sp_4$ with Levi factor $\GGL_1\times\sp_2$, and let $\pi=\rm{Ind}^{\sp_4\fq}_{P^F}(\sigma\otimes{\bf1})$ with $\sigma\ne {\bf1}$. Then $\pi$ is irreducible and the partition corresponding to the wavefront set of $\pi$ is $(2,2)$. But the partitions corresponding to the wavefront sets of ${\bf1}$ and $\rm{Ind}^{\sp_2\fq}_{P^{\prime F}}(\sigma)$ is $(1,1)$ and $(2)$, respectively.

\subsection{Case (A)}\label{sec6.1}

\begin{theorem}\label{a1}
Let $\pi$ be an irreducible representation of $\sp_{2n}\fq$ in Case (A) and $\pi'$ be an irreducible representation of $\sp_{2k}\fq$ with $n\ge k$.

(i) If $\pi'$ is in Case (A), then
\[
m_\psi(\pi,\pi')  =\prod_{[a]}
 m(\pi[a],\pi'[-a])
 \]
 where $\pi[a]$ and $\pi'[-a]$ are defined in
  (\ref{a}) and they are both unipotent representations of finite general linear groups or finite unitary groups, and $m(\pi[a],\pi'[-a])$ is defined in \cite[(1.2)]{Wang2} for unitary groups and in \cite[p.4]{Wang2} for general linear groups.

  (ii) Assume that $m_\psi(\pi,\pi')\ne 0$. Let $\sigma'$ be the unique irreducible representation  of $\sp_{2k'}\fq$ in Case (A) such that
  $
  \pi'[-a]=\sigma'[-a]
 $ for any $a\ne \pm 1$. Then
  \[
m_\psi(\pi,\pi')=m_\psi(\pi,\sigma')=\prod_{[a]\ne \pm 1}
 m(\pi[a],\pi'[-a]).
 \]
 In particular, each irreducible component of the descent of $\pi$ is an irreducible representation in Case (A).
\end{theorem}

\begin{proof}
We immediately get (ii) from (i) of this theorem and \cite[Theorem 1.4 (i)]{Wang1}.

 We prove (i) by theta correspondence and see-saw dual pairs, which are used in the proof of the Gan-Gross-Prasad conjecture (see \cite{Ato, GI, LW2,LW3,Wang1}). Since this theta argument is not used in other parts of this paper, we do not explain some notations, such as see-saw identity, the sign $\ee$, theta lifting $\Theta^{+}_{n,n',\psi}(\pi)$, and weil representation $\omega_{n,\psi}^{\ee}$, and readers can find them in \cite[section 5]{Wang1}. Here we follow the notations and definitions of \cite{Wang1}. The proof of this theorem is based on the proof of \cite[Proposition 7.8]{Wang1}.

By \cite[Proposition 7.3]{Wang1}, we have
\[
m_\psi(\pi,\pi')=\langle  \pi\otimes \omega_{n,\psi}^{\ee}, I_{P}^{\sp_{2n}}(\tau\otimes\pi)\rangle _{\sp_{2n}(\Fq)},
\]
where $P$ is an $F$-stable parabolic subgroup of $\sp_{2n}$ with Levi factor $\GGL_{n-k}\times\sp_{2k}$, $I_{P}^{\sp_{2n}}$ is the parabolic induction, and $\tau$ is a generic representation of $\GGL_{n-k}\fq$.
In the proof of \cite[Proposition 7.8]{Wang1}, we consider the see-saw diagram
\[
\setlength{\unitlength}{0.8cm}
\begin{picture}(20,5)
\thicklines
\put(6.5,4){$\sp_{2n}\times \sp_{2n}$}
\put(7.3,1){$\sp_{2n}$}
\put(12.3,4){$\o^{\ee}_{2n+1}$}
\put(11.6,1){$\o^{+}_{2n}\times \o^{\ee}_1$}
\put(7.7,1.5){\line(0,1){2.1}}
\put(12.8,1.5){\line(0,1){2.1}}
\put(8,1.5){\line(2,1){4.2}}
\put(8,3.7){\line(2,-1){4.2}}
\end{picture}
\]
where $\o^{\ee}_{2n+1}$, $\o^{+}_{2n}$, and $\o^{\ee}_1$ are some orthogonal groups defined in \cite[subsection 1.1]{Wang1}, and $\o^{+}_{2n}$ is a subgroup of $\o^{\ee}_{2n+1}$.
Let
$
\sigma:=\Theta^{+}_{n,n,\psi}(\pi)
$
be the image of theta correspondence of $\pi$ from $\sp_{2n}\fq$ to $\o^{+}_{2n}\fq$ with respect to character $\psi$, and
$
\sigma':=\Theta^{\ee}_{k,k,\psi}(\pi')
$
be the image of theta correspondence of $\pi'$ from $\sp_{2k}\fq$ to $\o^{\ee}_{2k+1}\fq$ with respect to character $\psi$.
As in the proof of \cite[Proposition 7.8]{Wang1}, we know that
\begin{itemize}
\item $\sigma$ and $\sigma'$ are irreducible representations;
\item the image of theta correspondence $\Theta^{+}_{n,n,\psi}(\sigma)$ of $\sigma$ from $\o^{+}_{2n}\fq$ to $\sp_{2n}\fq$ is $\pi$;
\item $\sigma(-I|_{\o^{+}_{2n}\fq})=I^{\o^{\ee}_{2n+1}}_{P'}(\tau\otimes\sigma')(-I)$, where $I$ is the identity element of $\o^{\ee}_{2n+1}\fq$ and $P'$ is an $F$-stable maximal parabolic subgroup of $\o^{\ee}_{2n+1}$ corresponding to $P$.
 \end{itemize}
Then by see-saw identity we have
\[
\langle  \pi\otimes \omega_{n,\psi}^{\ee}, I_{P}^{\sp_{2n}}(\tau\otimes\pi')\rangle _{\sp_{2n}(\Fq)}=\langle  \sigma,  I^{\o^{\ee}_{2n+1}}_{P'}(\tau\otimes\sigma')\rangle_{\o^+_{2n}(\Fq)}\quad \textrm{((7.5) in \cite{Wang1})}.
\]
It easy to check that $\sigma_{\so}:=\sigma|_{\so^+_{2n}(\Fq)}$ and $\sigma'_{\so}:=\sigma'|_{\so^{\ee}_{2k+1}\fq}$ are irreducible, and
\[
I^{\o^{\ee}_{2n+1}}_{P'}(\tau\otimes\sigma')|_{\so^{\ee}_{2n+1}\fq}=I^{\so^{\ee}_{2n+1}}_{P''}(\tau\otimes\sigma'_{\so}),\]
where $P''$ is an $F$-stable parabolic subgroup  of $\so^{\ee}_{2n+1}$.
 Recall that $\sigma(-I|_{\o^{+}_{2n}\fq})=I^{\o^{\ee}_{2n+1}}_{P'}(\tau\otimes\sigma')(-I)$, we get
\[
\langle  \sigma,  I^{\o^{\ee}_{2n+1}}_{P'}(\tau\otimes\sigma')\rangle_{\o^+_{2n}(\Fq)}=\langle  \sigma_{\so},  I^{\so^{\ee}_{2n+1}}_{P''}(\tau\otimes\sigma'_{\so})\rangle_{\so^+_{2n}(\Fq)}.
\]
It follows from \cite[Proposition 5.2]{LW3} and \cite[Theorem 1.3]{Wang2} that
\[
\langle  \sigma_{\so},  I^{\so^{\ee}_{2n+1}}_{P''}(\tau\otimes\sigma'_{\so})\rangle_{\so^+_{2n}(\Fq)}=m(\sigma_{\so},\sigma'_{\so})=\prod_{[a]\ne \pm 1}m(\sigma_{\so}[a],\sigma'_{\so}[a]),
\]
where $\prod_{[a]}\sigma_{\so}[a]$ and $\prod_{[a]}\sigma'_{\so}[a]$ are the image of $\sigma_{\so}$ and $\sigma'_{\so}$ via Lusztig correspondence, respectively (similar to (\ref{lc})). Note that $\sigma_{\so}[\pm1]$ and $\sigma'_{\so}[\pm1]$ are trivial representation of trivial group.

The relationship between Lusztig correspondence and theta correspondence has been established in \cite{P2}. By the main theorems in \cite[subsection 1.4]{P2}, we know that
\[
\pi[a]=\sigma_{\so}[a]
\]
and
\[
\pi'[-a]=\sigma'_{\so}[a],
\]
which completes the proof.
\end{proof}

   We emphasize that in the general linear group case, the multiplicity $ m(\pi[a],\pi'[-a])$ is {\emph not} the multiplicity in Gan-Gross-Prasad problem. As a consequence of Theorem \ref{a1}, in Case (A), the descent of an irreducible representation $\pi$ does not depend on the choice of $\psi$. To simplify notations, we write $\cal{D}^{\rm{FJ}}$ instead of $\cal{D}^{\rm{FJ}}_\psi$.

In our pervious work, we have calculated the descent of unipotent representations of unitary group in \cite{LW2}:

\begin{theorem}
Let $\lambda$ be a partition of $n$ into $k$ rows, and $\lambda^{-1}$ be the partition of $n-k$ obtained by removing the first column of $\lambda$. Let $\pi\in\rm{Irr}(\UU_{m}\fq)$. Then
\[
m(\pi_\lambda,\pi')=0
\]
if $m< n-k$. Moreover, for $m=n-k$,
\[
m(\pi_\lambda,\pi')=
\left\{
  \begin{aligned}
 1&\textrm{ if }\pi'=\pi_{\lambda^{-1}};\\
  0&\textrm{ otherwise}.
\end{aligned}
\right.
\]
\end{theorem}
By the same argument, similar result holds for general linear groups. This branching problem from $\GGL_n\fq$ to $\GGL_{n-1}\fq$ is also fully analysed in \cite{T}. Then we get the descent of $\pi$ as follows:

\begin{theorem}\label{d1}
Let $G=\sp_{2n}$ and let $\pi\in\cal{E}(G^F,s)$. Suppose that
\[
\cal{L}_s(\pi)=\prod_{[a]}\pi{[a]}
\]
 and $\pi[a]$ is corresponding to partition $\lambda[a]$ for each $[a]$. Write $\lambda[a]^t=(\lambda[a]^t_1,\lambda[a]^t_2,\cdots)$.
 Then
$\ell_0=\ell_0^\rm{FJ}(\pi)=\sum_{[a]}\#[a]\cdot\lambda[a]^t_1$ and
\[
 \CD^\rm{FJ}_{}(\pi)=\pi'
\]
and
\[
m(\pi,\pi')=1
\]
where the image of Lusztig correspondence of $\pi'$ is
\[
\prod_{[a]}\pi'[a],
\]
and for each $[a]$, $\pi'[a]$ is a unipotent representation corresponding to partition $\lambda[-a]^{-1}$. In particular, if $\pi'\in \cal{E}(G^{\prime F},s')$ and  $C_{G^{*F}}(s)$ is a product of general linear groups, so is $C_{G^{\prime*F}}(s')$.
\end{theorem}

We next turn to consider the descent sequence of $\pi$. For an irreducible representation $\pi$ in Case (A), by Theorem \ref{a1} (ii) and Theorem \ref{d1}, the descent $\CD^\rm{FJ}_{}(\pi)$ of a irreducible representation $\pi$ is also an irreducible representation of Case (A). Furthermore, there is unique descent sequence index of $\pi$, and we denote it by $(\ell_1,\ell_2,\cdots)$.

\begin{corollary}\label{da}
Keep notations as above. We have $\ell_i=\sum_{[a]}\#[a]\cdot\lambda[a]^t_i$. In particular, $\ell_i\ge \ell_{i-1}$ and $\hat\ell=(2\ell_1,2\ell_2,\cdots)$ is a partition.
\end{corollary}
\begin{proof}
It follows immediately from Theorem \ref{d1}.
\end{proof}

Then we discuss the relation between Deligne-Lusztig inductions and descents.
\begin{proposition}\label{dg1}
Let $\lambda$ be a partition of $n$ and $\pi_\lambda$ be the unipotent representation of $G=\GGL_n\fq$ (resp. $G=\UU_n\fq$). Let $L=\GGL_r\times\GGL_{n-r}$ (resp. $\UU_r\times\UU_{n-r}$) be a Levi subgroup of $\GGL_n$ (resp. $\UU_n$). Let $\sigma$ and $\tau$ be a irreducible representation of $\GGL_r\fq$ and $\GGL_{n-r}\fq$ (resp. $\UU_r\fq$ and $\UU_{n-r}\fq$), respectively. Suppose that $\sigma$ is generic. If
\begin{equation}\label{eq1}
\langle\pi_{\lambda},R^{G}_{L}(\sigma\otimes\tau)\rangle\ne 0
\end{equation}
then $r\le \lambda^t_1$.
\end{proposition}
\begin{proof}
We will only prove the general linear group case. The proof of the unitary group case is similar and will be left to the
reader.

Since $\pi_\lambda$ is unipotent, we can conclude that if (\ref{eq1}) holds,
then $\sigma$ and $\tau$ are both unipotent. So we can associate $\sigma$ and $\tau$ with two partitions $\mu$ and $\mu'$, respectively. Since $\sigma$ is generic, we get $\mu=(1,1,\cdots,1)$. This proposition follows from Littlewood-Richardson Rule and the discussion in \cite[section 3]{AM}.
\end{proof}

\begin{proposition}\label{prop6.6}
Let $G=\sp_{2n}$ and
let $\pi\in \cal{E}(G^F,s)$ in Case (A) such that $C_{G^{*F}}(s)$ is a product of general linear groups. Let $P$ be a parabolic subgroup of $G$ with Levi factor $\GGL_r\times\sp_{2(n-r)}$. Let $\sigma$ and $\tau$ be a irreducible representation of $\GGL_r\fq$ and $\sp_{2(n-r)}\fq$, respectively. Suppose that $\sigma$ is generic. If
\[
\langle\pi,\rm{Ind}^{G^F}_{P^F}(\sigma\otimes\tau)\rangle\ne 0
\]
then $r\le \ell_0^\rm{FJ}(\pi)=\sum_{[a]}\#[a]\cdot\lambda[a]^t_1$.
\end{proposition}
\begin{proof}
Note that $\pi$ is uniform, so in this case Lusztig correspondence is uniquely determined by Deligne-Lusztig characters. Moreover, Lusztig correspondence preserve the parabolic induction in uniform case. These facts allow us to prove this proposition on the dual group side. In other words,
\[
\langle\pi,R^{G}_{P}(\sigma\otimes\tau)\rangle=\langle\pi,\rm{Ind}^{G^F}_{P^F}(\sigma\otimes\tau)\rangle\ne 0
\]
if and only if
\[
\prod_{[a]}\langle\pi[a],\epsilon[a]R^{G_{[a]}(s)}_{L[a]}(\sigma[a]\otimes\tau[a])\rangle\ne 0
\]
where for each $[a]$, $\epsilon[a]\in\{\pm1\}$ depends on $G_{[a]}(s)$, and $L[a]=\GGL_{r[a]}\times\GGL_{n[a]-r[a]}$ is a certain Levi subgroup of $G_{[a]}(s)$, and $\sigma[a]$ is a generic representation. Then this proposition follows immediately from Proposition \ref{dg1}.

\end{proof}

\subsection{Case (B)}\label{sec6.2}
We have proved that for any irreducible representation $\pi$ in case (A), the descent sequence of $\pi$ is unique and the descent sequence index is a partition. But this is not always the case in case (B).

To begin with, let us recall the multiplicity in the Gan-Gross-Prasad problem in \cite{Wang1}, and calculate the descent of irreducible representations in Case (B).
Let
\[
\begin{aligned}
&\mathcal{G}^{\rm{even},+}_{n,m}:=\left\{(\Lambda,\Lambda')|
\Upsilon(\Lambda')_*\preccurlyeq\Upsilon(\Lambda)_*,\Upsilon(\Lambda)^*\preccurlyeq\Upsilon(\Lambda')^*,\rm{def}(\Lambda)>0,\rm{def}(\Lambda')=\rm{def}(\Lambda)-1\right\}; \\
&\mathcal{G}^{\rm{even},-}_{n,m}:=\left\{(\Lambda,\Lambda')|\Upsilon(\Lambda')_*\preccurlyeq\Upsilon(\Lambda)^*,
\Upsilon(\Lambda)_*\preccurlyeq\Upsilon(\Lambda')^*,\rm{def}(\Lambda)>0,\rm{def}(\Lambda')=-\rm{def}(\Lambda)-1\right\}; \\
&\mathcal{G}^{\rm{odd},-}_{n,m}:=\left\{(\Lambda,\Lambda')|\Upsilon(\Lambda')^*\preccurlyeq \Upsilon(\Lambda)^*,
\Upsilon(\Lambda)_*\preccurlyeq\Upsilon(\Lambda')_*,\rm{def}(\Lambda)<0,\rm{def}(\Lambda')=\rm{def}(\Lambda)+1\right\};\\
&\mathcal{G}^{\rm{odd},+}_{n,m}:=\left\{(\Lambda,\Lambda')|\Upsilon(\Lambda')^*\preccurlyeq \Upsilon(\Lambda)_*,
\Upsilon(\Lambda)^*\preccurlyeq\Upsilon(\Lambda')_*,\rm{def}(\Lambda)<0,\rm{def}(\Lambda')=-\rm{def}(\Lambda)+1\right\}
\end{aligned}
\]
be subsets of $\cal{S}_n\times\cal{S}_m^{\pm}$ where $\Upsilon(\Lambda)^*$ and $\Upsilon(\Lambda)_*$ are defined in subsection \ref{sec4.2}. Let
\[
\mathcal{G}=\bigcup_{n,m}\left(\mathcal{G}^{\rm{even},+}_{n,m}\bigcup\mathcal{G}^{\rm{even},-}_{n,m}
\bigcup\mathcal{G}^{\rm{odd},-}_{n,m}\bigcup\mathcal{G}^{\rm{odd},+}_{n,m}\right).
\]

\begin{theorem}[\cite{Wang1} Theorem 1.4 (i)]\label{u1}
Assume that $q$ is large enough such that the main theorem in \cite{S} holds.  Let $n\ge m$. Let $\prll\in\rm{Irr}(\sp_{2n}\fq)$ and $\pi_{\rho_1,\Lambda_1,\Lambda_1'}\in\rm{Irr}(\sp_{2m}\fq)$. Then we have
\[
\begin{aligned}
m_\psi(\prll,\pi_{\rho_1,\Lambda_1,\Lambda_1'})
=\left\{
\begin{array}{ll}
m_\psi(\pw_{\rho} ,\pw_{\rho_1}),  &\textrm{if }(\prll,\pi_{\rho_1,\Lambda_1,\Lambda_1'})\textrm{ is $(\psi,\ee)$-strongly relevant, and there}\\
&\textrm {are }
\widetilde{\Lambda_1'}\in\{\Lambda_1',\Lambda_1^{\prime t}\}
\textrm{ and }\widetilde{\Lambda'}\in\{\Lambda',\Lambda^{\prime t}\}\textrm{ such that }(\Lambda,\widetilde{\Lambda_1'})\\
&\textrm{and }(\Lambda_1,\wla')\in \mathcal{G};\\
0,&\textrm{otherwise}.
\end{array}\right.
\end{aligned}
\]
Here the definition of $(\psi,\ee)$-strongly relevant pair is in \cite[Section 6.3]{Wang1}.

\end{theorem}

\begin{remark}\label{ur}
In this paper, we need not know what $(\psi,\ee)$-strongly relevant pair actually is. Instead of its definition, we shall explain why we need this condition in Theorem \ref{u1}. Assume that there are $\widetilde{\Lambda_1'}\in\{\Lambda_1',\Lambda_1^{\prime t}\}$ and $\widetilde{\Lambda'}\in\{\Lambda',\Lambda^{\prime t}\}$ such that $(\Lambda,\widetilde{\Lambda_1'})$ and $(\Lambda_1,\wla')\in \mathcal{G}$. Since our parameterization of irreducible representations depends on the choice of Lusztig correspondence, we can not distinguish representations in the set
$
\{\pi_{\rho_1,\Lambda_1,\Lambda_1'},\pi_{\rho_1,\Lambda_1,\Lambda_1^{\prime t}}\}
$
 by our parameterization. However the parametrization of irreducible representations is not involved in the definition of $(\psi,\ee)$-strongly
relevant pair, so one can differ them by this condition of $(\psi,\ee)$-strongly
relevant pair. In fact, for a fixed $\prll$, there exists exactly one representation $\pi\in\{\pi_{\rho_1,\Lambda_1,\Lambda_1'},\pi_{\rho_1,\Lambda_1,\Lambda_1^{\prime t}}\}$ such that the pair $(\prll,\pi)$ is $(\psi,\ee)$-strongly relevant.
\end{remark}

\begin{corollary}\label{u2}
Assume that $q$ is large enough such that the main theorem in \cite{S} holds. Let $n\ge m$.

(i) Let $\pl\in\rm{Irr}(\sp_{2n}\fq)$ be an irreducible unipotent representation. For an irreducible representation $\prllc$ of $\sp_{2m}\fq$, we have
\[
m_\psi(\pl,\prllc)=\left\{
\begin{array}{ll}
1, &\textrm{if }(\pl,\prllc)\textrm{ is $(\psi,\ee)$-strongly relevant, and there is } \widetilde{\Lambda_1'}\in\{\Lambda_1',\Lambda_1^{\prime t}\}\\
&\textrm{such that }(\Lambda,\widetilde{\Lambda_1'})\in \mathcal{G} ,\textrm{ and } \pi_{\Lambda_1} \textrm{ and $\rho$ are generic};\\
0, & \textrm{otherwise.}
\end{array}\right.
\]

(ii) Let $\pi_{\Lambda'}\in\rm{Irr}(\sp_{2n}\fq)$ be an irreducible $\theta$-representation. For an irreducible representation $\prllc$ of $\sp_{2m}\fq$, we have
\[
m_\psi(\pi_{\Lambda'},\prllc)=\left\{
\begin{array}{ll}
1, &\textrm{if }(\pi_{\Lambda'},\prllc)\textrm{ is $(\psi,\ee)$-strongly relevant, and there is} \\
&\widetilde{\Lambda'}\in\{\Lambda',\Lambda^{\prime t}\}\textrm{ such that }(\Lambda_1,\wla')\in \mathcal{G} ,\textrm{ and } \pi_{\Lambda_1'} \textrm{ and $\rho$ are generic};\\
0, & \textrm{otherwise.}
\end{array}\right.
\]

\end{corollary}

\begin{corollary}\label{u3}
(i) Let $\pi$ be an irreducible unipotent representation of $\sp_{2n}\fq$. Then $\cal{D}^{\rm{FJ}}_\psi(\pi)$ is a $\theta$-representation of $\sp_{2m}\fq$ with $n\ge m$.

(ii) Let $\pi$ be an irreducible $\theta$-representation of $\sp_{2n}\fq$. Then there is $\psi$ such that $\cal{D}^{\rm{FJ}}_\psi(\pi)$ is a unipotent representation of $\sp_{2m}\fq$ with $n\ge m$ such that $\ell_{0,\psi}^{\rm{FJ}}(\pi)\ge\ell_{0,\psi'}^{\rm{FJ}}(\pi) $ for any character $ \psi'$ of $\Fq$.

\end{corollary}
\begin{proof}
We will only prove the case (i). The proof of the case (ii) is similar.

Assume that $\pi=\pl$ and $\prllc$ is a irreducible representation of $\sp_{2n'}\fq$ appearing in the descent of $\pi$. Suppose that $\prllc$ is not a $\theta$-representation. Then either $\rho$ or $\Lambda_1\ne 0$. Since $m_\psi(\pl,\prllc)\ne 0$, and by Corollary \ref{u2}, we have either $(\Lambda,\Lambda_1')\in \mathcal{G}$ or $(\Lambda,\Lambda_1^{\prime t})\in \mathcal{G}$. By Remark \ref{ur}, there exists a $\widetilde{\Lambda_1'}\in\{\Lambda_1',\Lambda_1^{\prime t}\}$ such that the pair $(\pl,\pi_{\widetilde{\Lambda_1'}})$ is $(\psi,\ee)$-strongly relevant. So we find an irreducible $\theta$-representation
$\pi_{\widetilde{\Lambda_1'}}$ of $\sp_{2n''}\fq$ with $n''< n'$ such that $m_\psi(\pl,\pi_{\widetilde{\Lambda_1'}})\ne 0$, which is a contradiction of definition of the descent.
\end{proof}

\begin{corollary}\label{u4}
Assume that $q$ is large enough such that the main theorem in \cite{S} holds. Let $n\ge m$.

(i) Let $\pl\in\rm{Irr}(\sp_{2n}\fq)$ be an irreducible unipotent representation. For an irreducible $\theta$-representation $\pi_{\Lambda'}$ of $\sp_{2m}\fq$, we have
\[
m_\psi(\pl,\pi_{\Lambda'})=\left\{
\begin{array}{ll}
1, &\textrm{if }(\pl,\pi_{\Lambda'})\textrm{ is $(\psi,\ee)$-strongly relevant and there is } \widetilde{\Lambda'}\in\{\Lambda',\Lambda^{\prime t}\}\\
&\textrm{such that }(\Lambda,\widetilde{\Lambda'})\in \mathcal{G};\\
0, & \textrm{otherwise.}
\end{array}\right.
\]

(ii) Let $\pi_{\Lambda'}\in\rm{Irr}(\sp_{2n}\fq)$ be an irreducible $\theta$-representation. For an irreducible unipotent representation $\pl$ of $\sp_{2m}\fq$, we have
\[
m_\psi(\pi_{\Lambda'},\pl)=\left\{
\begin{array}{ll}
1, &\textrm{if }(\pi_{\Lambda'},\pl)\textrm{ is $(\psi,\ee)$-strongly relevant, and there is } \widetilde{\Lambda'}\in\{\Lambda',\Lambda^{\prime t}\}\\
&\textrm{such that }(\Lambda,\wla')\in \mathcal{G};\\
0, & \textrm{otherwise.}
\end{array}\right.
\]

\end{corollary}
Recall that
\[
\cal{D}^{\rm{FJ}}_\psi(\pi)= \cal{D}^{\rm{FJ}}_{\psi'}(\pi)
\]
if  $\psi'$ is in the square class of $\psi$.
Since our aim of this subsection is to calculate the largest descent sequence of $\pl$, from now on, when we consider the descent of $\pl$, we always focus on this kind of descent $\CD^\rm{FJ}_{\ell_0, \psi}(\pl)$ with respect to $\psi$ such that the first occurrence index $\ell_{0,\psi}^{\rm{FJ}}(\pl)\ge\ell_{0,\psi'}^{\rm{FJ}}(\pl) $ for any character $ \psi'$ of $\Fq$.
Now assume that the character $\psi$ have above property and $\psi'$ is a character not in the square class of $\psi$. By above multiplicity results in the Gan-Gross-Prasad problem, for unipotent representation,
\[
\ell_{0,\psi}^{\rm{FJ}}(\pl)=\ell_{0,\psi'}^{\rm{FJ}}(\pl),
\]
and for $\theta$-representation,
\[
\ell_{0,\psi}^{\rm{FJ}}(\pl)>\ell_{0,\psi'}^{\rm{FJ}}(\pl).
\]
According to Remark \ref{ur}, Corollary \ref{u3} and Corollary \ref{u4}, an irreducible component of the descent $\cal{D}^{\rm{FJ}}_\psi(\pl)$ of an irreducible unipotent (resp. $\theta$) representation $\pl\in\rm{Irr}(\sp_{2n}\fq)$ is an irreducible $\theta$ (resp. unipotent) representation $\pi_{\Lambda'}\in \rm{Irr}(\sp_{2m}\fq)$ and the first occurrence descent index $\ell_{0,\psi}^{\rm{FJ}}(\pl)$ only depends on the set $\{\Lambda,\Lambda^t\}$, not on the symbol $\Lambda$ and the choice of $\psi$ with above property. Moreover, the set $\{\Lambda',\Lambda^{\prime t}\}$ also only depends on the set $\{\Lambda,\Lambda^t\}$, not on the symbol $\Lambda$ and $\psi$. Furthermore, consider a largest descent sequence of $\pi$:
  \[
\pi_1=\pl\xrightarrow{\ell_1}\pi_2\xrightarrow{\ell_2}\pi_{3}\xrightarrow{\ell_3}\cdots\xrightarrow{\ell_k}\bf{1}.
\]
such that in each step, the character $\psi_i$ in the descent $\CD^\rm{FJ}_{\ell_0, \psi_i}(\pi_i)$ has above property. Note that $\psi_i$ may be different. We observed that
the descent sequence index $(2\ell_i)$ only depends on the set $\{\Lambda,\Lambda^t\}$. We will only care about $\ell_i$. And they only depends on the set $\{\Lambda,\Lambda^t\}$ not $\psi_i$. So from now on, we write $\cal{D}^{\rm{FJ}}$ instead of $\cal{D}^{\rm{FJ}}_\psi$.

By above multiplicity results, we can describe the descent of $\pl$ by the symbol $\Lambda$ as follow:
For $\rm{def}(\Lambda)>0,$ we set
 \begin{itemize}
\item $\mathcal{DG}^{+}_{\rm{un}}(\Lambda):=\Lambda'$ such that $\Upsilon(\Lambda')_*=\Upsilon(\Lambda)^{*-1}$, $\Upsilon(\Lambda')^*=\Upsilon(\Lambda)_*$ and $\rm{def}(\Lambda')=-\rm{def}(\Lambda)-1$;
\item $\mathcal{DG}^{-}_{\rm{un}}(\Lambda):=\Lambda'$ such that $\Upsilon(\Lambda')_*=\Upsilon(\Lambda)^{-1}_*$, $\Upsilon(\Lambda')^*=\Upsilon(\Lambda)^*$ and $\rm{def}(\Lambda')=\rm{def}(\Lambda)-1$;
  \end{itemize}
For $\rm{def}(\Lambda)<0,$ we set
 \begin{itemize}
 \item  $\mathcal{DG}^{+}_{\rm{un}}(\Lambda):=\Lambda'$ such that $\Upsilon(\Lambda')^*= \Upsilon(\Lambda)^{-1}_*$, $\Upsilon(\Lambda')_*=\Upsilon(\Lambda)^*$ and $\rm{def}(\Lambda')=-\rm{def}(\Lambda)+1$.
\item $\mathcal{DG}^{-}_{\rm{un}}(\Lambda):=\Lambda'$ such that $\Upsilon(\Lambda')^*= \Upsilon(\Lambda)^{*-1}$, $\Upsilon(\Lambda')_*=\Upsilon(\Lambda)_*$ and $\rm{def}(\Lambda')=\rm{def}(\Lambda)+1$;
\end{itemize}
For $\rm{def}(\Lambda')\ge0,$ we set
 \begin{itemize}

\item $\widetilde{\mathcal{DG}^{+}_{\theta}}(\Lambda'):=\Lambda$ such that $\Upsilon(\Lambda)_*=\Upsilon(\Lambda')_*$, $\Upsilon(\Lambda)^*=\Upsilon(\Lambda')^{*-1}$ and $\rm{def}(\Lambda)=\rm{def}(\Lambda')+1$;
\item $\widetilde{\mathcal{DG}^{-}_{\theta}}(\Lambda'):=\Lambda$ such that $\Upsilon(\Lambda)_*=\Upsilon(\Lambda')^*$, $\Upsilon(\Lambda)^*=\Upsilon(\Lambda')^{-1}_*$ and $\rm{def}(\Lambda)=-\rm{def}(\Lambda')+1$;
  \end{itemize}
    For $\rm{def}(\Lambda')<0,$ we set
 \begin{itemize}
 \item  $\widetilde{\mathcal{DG}^{+}_{\theta}}(\Lambda'):=\Lambda$ such that $\Upsilon(\Lambda)^*=\Upsilon(\Lambda')^*$, $\Upsilon(\Lambda)_*=\Upsilon(\Lambda')^{-1}_*$ and $\rm{def}(\Lambda)=\rm{def}(\Lambda')-1$.
\item $\widetilde{\mathcal{DG}^{-}_{\theta}}(\Lambda'):=\Lambda$ such that $\Upsilon(\Lambda)^*=\Upsilon(\Lambda')_*$, $\Upsilon(\Lambda)_*=\Upsilon(\Lambda')^{*-1}$ and $\rm{def}(\Lambda)=-\rm{def}(\Lambda')-1$;
\end{itemize}
Note that if $\Lambda'\in\bigcup_n\cal{S}_n^\pm$, then there is unique $\Lambda^{\prime \epsilon}\in\{\Lambda^{\prime },\Lambda^{\prime t}\}$ such that $\widetilde{\mathcal{DG}^{\epsilon}_{\theta}}(\Lambda^{\prime \epsilon})\in\bigcup_n\cal{S}_n$. In fact, the existence of $\Lambda^{\prime \epsilon}$ is easy to check, and for the uniqueness, if $\widetilde{\mathcal{DG}^{\epsilon}_{\theta}}(\Lambda^{\prime \epsilon})\in\bigcup_n\cal{S}_n$, then
$\widetilde{\mathcal{DG}^{\epsilon}_{\theta}}((\Lambda^{\prime \epsilon})^t)=\left(\widetilde{\mathcal{DG}^{\epsilon}_{\theta}}(\Lambda^{\prime \epsilon})\right)^t$. So
$\rm{def}\left(\widetilde{\mathcal{DG}^{\epsilon}_{\theta}}((\Lambda^{\prime \epsilon})^t)\right)=3\textrm{ mod }4$ and $\widetilde{\mathcal{DG}^{\epsilon}_{\theta}}((\Lambda^{\prime \epsilon})^t)\notin\bigcup_n\cal{S}_n$.
For $\Lambda'\in\bigcup_n\cal{S}_n^\pm$, we set
\[
\mathcal{DG}^{\epsilon}_{\theta}(\Lambda')=\widetilde{\mathcal{DG}^{\epsilon}_{\theta}}(\Lambda^{\prime \epsilon}).
\]

\begin{theorem}[descent]\label{u5}

(i) Let $\pi_{\Lambda}\in\rm{Irr}(\sp_{2n}\fq)$ be an irreducible unipotent representation and $k=\frac{|\rm{def}(\Lambda)|-1}{2}$. Then
\[
\cal{D}^{\rm{FJ}}(\pi_\Lambda):=\left\{
\begin{array}{ll}
\pi_{\Lambda^\epsilon} &\textrm{ if }\rm{rank}(\cal{DG}_{\rm{un}}^\epsilon(\Lambda))<\rm{rank}(\cal{DG}_{un}^{-\epsilon}(\Lambda));\\
\pi_{\Lambda^+}+\pi_{\Lambda^-}&\textrm{ if }\rm{rank}(\cal{DG}_{\rm{un}}^\epsilon(\Lambda))=\rm{rank}(\cal{DG}_{\rm{un}}^{-\epsilon}(\Lambda)),
\end{array}\right.
\]
where $\pi_{\Lambda^\epsilon}$ is an irreducible $\theta$-representation and $\Lambda^\epsilon\in\{\cal{DG}_{\rm{un}}^\epsilon(\Lambda),(\cal{DG}_{\rm{un}}^\epsilon(\Lambda))^t\}$
such that $(\pl,\pi_{\Lambda^\epsilon})$ is $(\psi,\ee)$-strongly relevant.
And
\[
\ell^{\rm{FJ}}_0(\pi_\Lambda):=\left\{
\begin{array}{ll}
\left(\Upsilon(\Lambda\right)^{*})^t_1-(k+1) &\textrm{ if }\rm{def}(\Lambda)>0\textrm{ and }\rm{rank}(\cal{DG}_{\rm{un}}^+(\Lambda))\le\rm{rank}(\cal{DG}_{\rm{un}}^{-}(\Lambda));\\
\left(\Upsilon(\Lambda\right)_{*})^t_1+k &\textrm{ if }\rm{def}(\Lambda)>0\textrm{ and }\rm{rank}(\cal{DG}_{\rm{un}}^+(\Lambda))>\rm{rank}(\cal{DG}_{\rm{un}}^{-}(\Lambda));\\
\left(\Upsilon(\Lambda\right)_{*})^t_1-(k+1) &\textrm{ if }\rm{def}(\Lambda)<0\textrm{ and }\rm{rank}(\cal{DG}_{\rm{un}}^+(\Lambda))\le\rm{rank}(\cal{DG}_{\rm{un}}^{-}(\Lambda));\\
\left(\Upsilon(\Lambda\right)^{*})^t_1+k &\textrm{ if }\rm{def}(\Lambda)<0\textrm{ and }\rm{rank}(\cal{DG}_{\rm{un}}^+(\Lambda))>\rm{rank}(\cal{DG}_{\rm{un}}^{-}(\Lambda)).\\
\end{array}\right.
\]
If $\rm{rank}(\cal{DG}_{A}^+(\Lambda))=\rm{rank}(\cal{DG}_{A}^{-}(\Lambda))$, then
\[
\left\{
\begin{array}{ll}
\left(\Upsilon(\Lambda\right)^{*})^t_1-(k+1) =\left(\Upsilon(\Lambda\right)_{*})^t_1+k &\textrm{ if }\rm{def}(\Lambda)>0;\\
\left(\Upsilon(\Lambda\right)_{*})^t_1-(k+1)=\left(\Upsilon(\Lambda\right)^{*})^t_1+k &\textrm{ if }\rm{def}(\Lambda)<0.
\end{array}\right.
\]

(ii) Let $\pi_{\Lambda'}\in\rm{Irr}(\sp_{2n}\fq)$ be an irreducible $\theta$-representation and $k=\frac{|\rm{def}(\Lambda')|}{2}$. Then
\[
\cal{D}^{\rm{FJ}}(\pi_{\Lambda'}):=\left\{
\begin{array}{ll}
\pi_{\cal{DG}_{\theta}^\epsilon(\Lambda')} &\textrm{ if }\rm{rank}(\cal{DG}_{\theta}^\epsilon(\Lambda'))<\rm{rank}(\cal{DG}_{\theta}^{-\epsilon}(\Lambda'));\\
\pi_{\cal{DG}_{\theta}^+(\Lambda')}+\pi_{\cal{DG}_{\theta}^-(\Lambda')}&\textrm{ if }\rm{rank}(\cal{DG}_{\theta}^\epsilon(\Lambda'))=\rm{rank}(\cal{DG}_{\theta}^{-\epsilon}(\Lambda')),
\end{array}\right.
\]
where $\pi_{\cal{DG}_{\theta}^\epsilon(\Lambda')} $ is an irreducible unipotent representation
and
\[
\ell^{\rm{FJ}}_0(\pi_{\Lambda'}):=\left\{
\begin{array}{ll}
\left(\Upsilon(\Lambda'\right)^{*})^t_1-k &\textrm{ if }\rm{def}(\Lambda')\ge0\textrm{ and }\rm{rank}(\cal{DG}_{\theta}^+(\Lambda'))\le\rm{rank}(\cal{DG}_{\theta}^{-}(\Lambda'));\\
\left(\Upsilon(\Lambda'\right)_{*})^t_1+k &\textrm{ if }\rm{def}(\Lambda')\ge0\textrm{ and }\rm{rank}(\cal{DG}_{\theta}^+(\Lambda'))>\rm{rank}(\cal{DG}_{\theta}^{-}(\Lambda'));\\
\left(\Upsilon(\Lambda'\right)_{*})^t_1-k &\textrm{ if }\rm{def}(\Lambda')<0\textrm{ and }\rm{rank}(\cal{DG}_{\theta}^+(\Lambda'))\le\rm{rank}(\cal{DG}_{\theta}^{-}(\Lambda'));\\
\left(\Upsilon(\Lambda'\right)^{*})^t_1+k &\textrm{ if }\rm{def}(\Lambda')<0\textrm{ and }\rm{rank}(\cal{DG}_{\theta}^+(\Lambda'))>\rm{rank}(\cal{DG}_{\theta}^{-}(\Lambda')).\\
\end{array}\right.
\]
If $\rm{rank}(\cal{DG}_{\theta}^+(\Lambda'))=\rm{rank}(\cal{DG}_{\theta}^{-}(\Lambda'))$, then
\[
\left\{
\begin{array}{ll}
\left(\Upsilon(\Lambda\right)^{*})^t_1-k =\left(\Upsilon(\Lambda\right)_{*})^t_1+k &\textrm{ if }\rm{def}(\Lambda)>0;\\
\left(\Upsilon(\Lambda\right)_{*})^t_1-k=\left(\Upsilon(\Lambda\right)^{*})^t_1+k &\textrm{ if }\rm{def}(\Lambda)<0.
\end{array}\right.
\]
\end{theorem}
\begin{proof}

For (i), by Corollary \ref{u3}, we know that $\cal{D}^{\rm{FJ}}(\pi_{\Lambda})$ is a $\theta$-representation (it may not be irreducible). All that remains is to pick certain irreducible $\theta$-representations, or equivalently, certain symbols $\Lambda'$ which satisfies the conditions in Theorem \ref{u4} with the smallest rank. By directly calculation, the symbol is either $\cal{DG}_{\rm{un}}^+(\Lambda)$ or $\cal{DG}_{\rm{un}}^-(\Lambda)$. Then we pick the symbol $\Lambda'$ by comparing their ranks. The calculation of $\ell^{\rm{FJ}}_0(\pi_\Lambda)$ is straightforward.

 For (ii), one should note that  $\rm{rank}(\mathcal{DG}^{\epsilon}_{\theta}(\Lambda'))=\rm{rank}(\widetilde{\mathcal{DG}^{\epsilon}_{\theta}}(\Lambda'))=\rm{rank}(\widetilde{\mathcal{DG}^{\epsilon}_{\theta}}(\Lambda^{\prime t}))$ and $\ell^{\rm{FJ}}_0(\pi_{\Lambda'})=\ell^{\rm{FJ}}_0(\pi_{\Lambda^{\prime t}})$. The rest of the proof of (ii) is similar to (i).
\end{proof}

In Example \ref{ue}, we known that a series of descents may not produce a partition. But in most cases, we will get a partition by descent method, and even if not, the array that we get by descent method is close to certain partition.

\begin{proposition}\label{u5.5}
(i) Let $\pi_{\Lambda}\in\rm{Irr}(\sp_{2n}\fq)$ be an irreducible unipotent representation and $k=\frac{|\rm{def}(\Lambda)|-1}{2}$. Assume that $\pi_{\Lambda'}\in \rm{Irr}(\sp_{2n'}\fq)$ is an irreducible $\theta$-representation such that $\pi_{\Lambda}$ appears in $\cal{D}^{\rm{FJ}}(\pi_{\Lambda'})$. Then $\ell^{\rm{FJ}}_0(\pi_\Lambda)\le \ell^{\rm{FJ}}_0(\pi_{\Lambda'})$ with equality hold if and only if $\rm{rank}(\cal{DG}_{\theta}^\epsilon(\Lambda'))=\rm{rank}(\cal{DG}_{\theta}^{-\epsilon}(\Lambda'))$. In other words, the equality holds if and only if
\begin{equation}\label{equ}
\left\{
\begin{array}{ll}
\left(\Upsilon(\Lambda'\right)^{*})^t_1-\frac{|\rm{def}(\Lambda)|-1}{2}= \left(\Upsilon(\Lambda'\right)_{*})^t_1+\frac{|\rm{def}(\Lambda)|-1}{2} &\textrm{ if } \rm{def}(\Lambda)\ge 0;\\
\left(\Upsilon(\Lambda'\right)_{*})^t_1-\frac{|\rm{def}(\Lambda)|-1}{2}= \left(\Upsilon(\Lambda'\right)^{*})^t_1+\frac{|\rm{def}(\Lambda)|-1}{2}&\textrm{ if } \rm{def}(\Lambda)< 0.
\end{array}\right.
\end{equation}

(ii) Let $\pi_{\Lambda'}\in\rm{Irr}(\sp_{2n}\fq)$ be an irreducible $\theta$-representation. Assume that $\pi_{\Lambda}\in \rm{Irr}(\sp_{2n'}\fq)$ is an irreducible unipotent representation such that $\pi_{\Lambda'}$ appears in $\cal{D}^{\rm{FJ}}(\pi_{\Lambda})$. Then $\ell^{\rm{FJ}}_0(\pi_{\Lambda'})\le \ell^{\rm{FJ}}_0(\pi_{\Lambda})+1$ with equality holds if and only if $\rm{rank}(\cal{DG}_{\rm{un}}^\epsilon(\Lambda))=\rm{rank}(\cal{DG}_{\rm{un}}^{-\epsilon}(\Lambda))$. In particular, we have
\[
\Lambda=\begin{pmatrix}
a_1,a_2,\cdots,a_{m+2k+1}\\
b_1,b_2,\cdots,b_{m}
\end{pmatrix},
\]
where $a_{m+2k+1}>0$ and $b_m>0$.

\end{proposition}
\begin{proof}
We will only prove the (i). The proof of the case (ii) is similar.

Note that  $\mathcal{DG}^{\epsilon}_{\theta}(\Lambda')=\mathcal{DG}^{\epsilon}_{\theta}(\Lambda^{\prime t})$ and $\rm{def}(\Lambda')=-\rm{def}(\Lambda^{\prime t})$. We only prove this proposition for $\rm{def}(\Lambda')\ge0$. By Corollary \ref{u4} and Theorem \ref{u5}, there are four possibilities:
 \begin{itemize}
 \item $\rm{def}(\Lambda)=\rm{def}(\Lambda')+1$;
 \item $\rm{def}(\Lambda)=\rm{def}(\Lambda')-1$;
\item $\rm{def}(\Lambda)=-\rm{def}(\Lambda')+1$;
 \item $\rm{def}(\Lambda)=-\rm{def}(\Lambda')-1$.
\end{itemize}
And $\mathcal{DG}^{\epsilon}_{\theta}(\Lambda')=\widetilde{\mathcal{DG}^{\epsilon}_{\theta}}(\Lambda^{\prime })$ for the first case and for the second case with $\rm{def}(\Lambda')>0$ while $\mathcal{DG}^{\epsilon}_{\theta}(\Lambda')=\widetilde{\mathcal{DG}^{\epsilon}_{\theta}}(\Lambda^{\prime t})$ for the last two and for the second case with $\rm{def}(\Lambda')=0$. We will only prove the first two cases.

In the first case, we have $\rm{def}(\Lambda)=\rm{def}(\Lambda')+1$ and $\frac{|\rm{def}(\Lambda')|}{2}=k$. Since $\pi_{\Lambda}$ appears in $\cal{D}^{\rm{FJ}}(\pi_{ \Lambda'})$, by Theorem \ref{u5} (ii), we know that $\pl=\pi_{\cal{DG}_{\theta}^+(\Lambda')}$ and $\rm{rank}(\cal{DG}_{\theta}^+(\Lambda'))\le\rm{rank}(\cal{DG}_{\theta}^{-}(\Lambda'))$. It is equivalent to
\[
\ell^{\rm{FJ}}_0(\pi_{ \Lambda'})=\left(\Upsilon(\Lambda'\right)^{*})^t_1-k\ge \left(\Upsilon(\Lambda'\right)_{*})^t_1+k
\]
and $\Upsilon(\Lambda)_*=\Upsilon(\Lambda')_*$, $\Upsilon(\Lambda)^*=\Upsilon(\Lambda')^{*-1}$. Then we have
\[
\left(\Upsilon(\Lambda\right)_{*})^t_1+k=\left(\Upsilon(\Lambda'\right)_{*})^t_1+k \le\ell^{\rm{FJ}}_0(\pi_{ \Lambda'})
\]
and
\[
\left(\Upsilon(\Lambda\right)^{*})^t_1-(k+1)=\left(\Upsilon(\Lambda'\right)^{*})^t_2-(k+1)<\left(\Upsilon(\Lambda'\right)^{*})^t_1-k =\ell^{\rm{FJ}}_0(\pi_{ \Lambda'}),
\]\
which implies that $\ell^{\rm{FJ}}_0(\pi_\Lambda)=\rm{max}\{\left(\Upsilon(\Lambda\right)_{*})^t_1+k,\left(\Upsilon(\Lambda\right)^{*})^t_1-(k+1)\}\le \ell^{\rm{FJ}}_0(\pi_{ \Lambda'})$ with equality hold if and only if $\left(\Upsilon(\Lambda'\right)^{*})^t_1-k= \left(\Upsilon(\Lambda'\right)_{*})^t_1+k$.

Assume that $\rm{def}(\Lambda)=\rm{def}(\Lambda')-1$. We have $\frac{|\rm{def}(\Lambda')|}{2}=k+1$ with $k\ge-1$. We only prove for $k\ge 0$, and the $k=-1$ case is similar. With same argument, we know that $\pl=\pi_{\cal{DG}_{\theta}^-(\Lambda')}$ and $\rm{rank}(\cal{DG}_{\theta}^-(\Lambda'))\le\rm{rank}(\cal{DG}_{\theta}^{+}(\Lambda'))$. It is equivalent to
\[
\ell^{\rm{FJ}}_0(\pi_{ \Lambda'})=\left(\Upsilon(\Lambda'\right)_{*})^t_1+(k+1)\ge \left(\Upsilon(\Lambda'\right)^{*})^t_1-(k+1)
\]
and $\Upsilon(\Lambda)_*=\Upsilon(\Lambda')_{*}^{-1}$, $\Upsilon(\Lambda)^*=\Upsilon(\Lambda')^*$. Then we have
\[
\left(\Upsilon(\Lambda\right)_{*})^t_1+k=\left(\Upsilon(\Lambda'\right)_{*})^t_2+k<\left(\Upsilon(\Lambda'\right)_{*})^t_1+(k+1) =\ell^{\rm{FJ}}_0(\pi_{ \Lambda'})
\]
and
\[
\left(\Upsilon(\Lambda\right)^{*})^t_1-(k+1)=\left(\Upsilon(\Lambda'\right)^{*})^t_1-(k+1) \le\ell^{\rm{FJ}}_0(\pi_{ \Lambda'}),
\]\
which implies that $\ell^{\rm{FJ}}_1(\pi_\Lambda)=\rm{max}\{\left(\Upsilon(\Lambda\right)_{*})^t_1+k,\left(\Upsilon(\Lambda\right)^{*})^t_1-(k+1)\}\le \ell^{\rm{FJ}}_1(\pi_{ \Lambda'})$ with equality hold if and only if $\left(\Upsilon(\Lambda'\right)^{*})^t_1-(k+1)= \left(\Upsilon(\Lambda'\right)_{*})^t_1+(k+1)$.

\end{proof}

\begin{proposition}\label{u7}
Let
 \[
\pi=\pi_1\xrightarrow{\ell_1}\pi_2\xrightarrow{\ell_2}\pi_3\cdots\xrightarrow{\ell_m-1}\pi_m
\]
be descent sequence of $\pi$. If $\pi$ is an irreducible unipotent (resp. $\theta$-) representation, then $\ell_i\ge \ell_{j}$ for $\forall i$ and $j\ge i+2$.
\end{proposition}
\begin{proof}
Let $\pi_i\in\rm{Irr}(\sp_{2n_i}\fq)$. We will only prove the $\theta$-representation case. The proof of the unipotent representation case is similar.
We prove it by induction on $m$ of both unipotent representation case and $\theta$-representation case.

If $m\le3$, then there is nothing to prove. Assume that $m>3$. For $i>1$, we set $\pi'_1:=\pi_i$ and consider the descent sequence
\[
\pi_1'=\pi_i\xrightarrow{\ell_1'=\ell_i}\pi'_2=\pi_{i+1}\xrightarrow{\ell_2'=\ell_{i+1}}\cdots\xrightarrow{\ell_{m-i}'=\ell_{m-1}}\pi'_{m-i+1}=\pi_m.
\]
By our induction assumption, we get $\ell_i\ge \ell_{j}$ for $\forall i$ and $j\ge i+2$.

For $i=1$, by Proposition \ref{u5.5} (i), we have $\ell_1\ge\ell_2$ and then $\ell_1\ge \ell_j$ with $j> 3$.
The rest is to prove $\ell_1\ge\ell_3$. If $\ell_1>\ell_2$, then applying Proposition \ref{u5.5} (ii) on $\pi_3$, we have $\ell_3\le \ell_2+1$. Hence $\ell_1\ge\ell_3$.

Assume that $\ell_1=\ell_2$, $\pi_1=\pi_{\Lambda'}$, $\pi_2=\pl$ and $k=\frac{|\rm{def}(\Lambda')|}{2}$ . Then applying Proposition \ref{u5.5} (i) on $\pi_2$, one has
\[
\pl=\left\{
\begin{array}{ll}
\pi_{\cal{D}^+_{\theta}(\Lambda')}&\textrm{ if } \frac{|\rm{def}(\Lambda)|-1}{2}=k;\\
\pi_{\cal{D}^-_{\theta}(\Lambda')}&\textrm{ if } \frac{|\rm{def}(\Lambda)|-1}{2}=k-1.
\end{array}\right.
\]
Set $\pi_3=\pi_{\Lambda''}$. Then by Theorem \ref{u5}, either $\frac{|\rm{def}(\Lambda'')|}{2}=\frac{|\rm{def}(\Lambda)|-1}{2}+1$ or $\frac{|\rm{def}(\Lambda'')|}{2}=\frac{|\rm{def}(\Lambda)|-1}{2}$.

Assume that $\frac{|\rm{def}(\Lambda'')|}{2}=\frac{|\rm{def}(\Lambda)|-1}{2}+1$. By Theorem \ref{u5}, we have
\[
\begin{aligned}
\ell_3=&\ell^{\rm{FJ}}_0(\pi_{\Lambda''})\\
=&\left\{
\begin{array}{ll}
\left(\Upsilon(\Lambda''\right)_{*})^t_1+\frac{|\rm{def}(\Lambda'')|}{2} &\textrm{ if }\rm{def}(\Lambda'')\ge0;\\
\left(\Upsilon(\Lambda''\right)^{*})^t_1+\frac{|\rm{def}(\Lambda'')|}{2} &\textrm{ if }\rm{def}(\Lambda'')<0;\\
\end{array}\right.\\
=&\left\{
\begin{array}{ll}
\left(\Upsilon(\Lambda\right)^{*})^t_1+\frac{|\rm{def}(\Lambda)|-1}{2}+1 &\textrm{ if }\rm{def}(\Lambda)<0;\\
\left(\Upsilon(\Lambda\right)_{*})^t_1+\frac{|\rm{def}(\Lambda)|-1}{2}+1 &\textrm{ if }\rm{def}(\Lambda)>0.\\
\end{array}\right.\\
\end{aligned}
\]
For $\pl=\pi_{\cal{D}^-_{\theta}(\Lambda')}$, by Theorem \ref{u5}, we have
\[
\begin{aligned}
\ell_3=&\left\{
\begin{array}{ll}
\left(\Upsilon(\Lambda\right)^{*})^t_1+\frac{|\rm{def}(\Lambda)|-1}{2}+1 &\textrm{ if }\rm{def}(\Lambda)<0;\\
\left(\Upsilon(\Lambda\right)_{*})^t_1+\frac{|\rm{def}(\Lambda)|-1}{2}+1 &\textrm{ if }\rm{def}(\Lambda)>0.\\
\end{array}\right.\\
=&\left\{
\begin{array}{ll}
\left(\Upsilon(\Lambda'\right)^{*})^t_2+\frac{|\rm{def}(\Lambda')|}{2} &\textrm{ if }\rm{def}(\Lambda')<0;\\
\left(\Upsilon(\Lambda'\right)_{*})^t_2+\frac{|\rm{def}(\Lambda')|}{2} &\textrm{ if }\rm{def}(\Lambda')\ge0.\\
\end{array}\right.\\
\le & \ell^{\rm{FJ}}_0(\pi_{\Lambda'})\\
=&\ell_1
\end{aligned}
\]
For $\pl=\pi_{\cal{D}^+_{\theta}(\Lambda')}$, by Theorem \ref{u5} (i) and Proposition \ref{u5.5} (ii), one has
\[
\begin{aligned}
\ell_3=&\left\{
\begin{array}{ll}
\left(\Upsilon(\Lambda\right)^{*})^t_1+\frac{|\rm{def}(\Lambda)|-1}{2}+1 &\textrm{ if }\rm{def}(\Lambda)<0;\\
\left(\Upsilon(\Lambda\right)_{*})^t_1+\frac{|\rm{def}(\Lambda)|-1}{2}+1 &\textrm{ if }\rm{def}(\Lambda)>0.\\
\end{array}\right.\\
=&\left\{
\begin{array}{ll}
\left(\Upsilon(\Lambda\right)_{*})^t_1-\frac{|\rm{def}(\Lambda)|-1}{2} &\textrm{ if }\rm{def}(\Lambda)<0;\\
\left(\Upsilon(\Lambda\right)^{*})^t_1-\frac{|\rm{def}(\Lambda)|-1}{2} &\textrm{ if }\rm{def}(\Lambda)>0.\\
\end{array}\right.\\
=&\left\{
\begin{array}{ll}
\left(\Upsilon(\Lambda'\right)_{*})^t_2-\frac{|\rm{def}(\Lambda')|}{2} &\textrm{ if }\rm{def}(\Lambda')<0;\\
\left(\Upsilon(\Lambda'\right)^{*})^t_2-\frac{|\rm{def}(\Lambda')|}{2} &\textrm{ if }\rm{def}(\Lambda')\ge0.\\
\end{array}\right.\\
\le & \ell^{\rm{FJ}}_0(\pi_{\Lambda'})\\
=&\ell_1
\end{aligned}
\]
The proof of $\frac{|\rm{def}(\Lambda'')|}{2}=\frac{|\rm{def}(\Lambda)|-1}{2}$ case is similar.

\end{proof}

Let $\hat{\ell}{(\pi)}=(2\ell_1,2\ell_2,\cdots)$ be the largest descent sequence of $\pl$ and it is an array. By Proposition \ref{u5.5} and Proposition \ref{u7}, there are three possibilities of $\ell_i$:
 \begin{itemize}
 \item $\ell_{i-1}\ge\ell_i\ge\ell_{i+1}$;
\item $\ell_i=\ell_{i+1}-1$ and $\ell_{i-1}>\ell_{i}$;
 \item $\ell_i=\ell_{i-1}+1$ and $\ell_{i+1}<\ell_i$;
\end{itemize}

 We set
\[
\widetilde{\ell}{(\pl)}_i:=\left\{
\begin{array}{ll}
2\ell_i, &  \textrm{if }\ell_{i-1}\ge\ell_i\ge\ell_{i+1};\\
2\ell_i+1, & \textrm{if }\ell_i=\ell_{i+1}-1;\\
2\ell_i-1, & \textrm{if }\ell_i=\ell_{i-1}+1.
\end{array}\right.
\]
\begin{proposition}
The array $\widetilde{\ell}{(\pl)}=(\widetilde{\ell}{(\pl)}_1,\widetilde{\ell}{(\pl)}_2,\cdots)$ is a partition.
\end{proposition}
\begin{proof}We have to prove that $\widetilde{\ell}{(\pl)}_i\ge \widetilde{\ell}{(\pl)}_{i+1}$.

Assume that $\ell_{i-1}\ge\ell_i\ge\ell_{i+1}$. If $\ell_{i+1}\ne\ell_{i+2}-1$, then $\widetilde{\ell}{(\pl)}_i=2\ell_i\ge2\ell_{i+1}\ge \widetilde{\ell}{(\pl)}_{i+1}$. If $\ell_{i+1}=\ell_{i+2}-1$, by Proposition \ref{u7}, $\ell_i\ge\ell_{i+2}$, then $\widetilde{\ell}{(\pl)}_i=2\ell_i\ge 2\ell_{i+1}+2>2\ell_{i+1}+1=\widetilde{\ell}{(\pl)}_{i+1}$.

Assume that $\ell_i=\ell_{i+1}-1$. In this case, we have $\widetilde{\ell}{(\pl)}_i=\widetilde{\ell}{(\pl)}_{i+1}$.
Assume that $\ell_i=\ell_{i-1}+1$. By Proposition \ref{u7}, we have $\ell_i>\ell_{i-1}\ge\ell_{i+1}$. So $\widetilde{\ell}{(\pl)}_i=2\ell_i-1\ge \widetilde{\ell}{(\pl)}_{i+1}$.
\end{proof}

We call the partition $\widetilde{\ell}{(\pl)}$ the descent partition of $\pl$. To simplify the notation, we denote $\widetilde{\ell}{(\pl)}$ by $\widetilde{\ell}{}=(\widetilde{\ell}{_i})$ when no confusion raise.

\subsection{Case (C)}
In this section, we consider the cuspidal case. With similar argument in Case (B), according to the main theorems in \cite{Wang1,Wang2}, we have following result.
\begin{theorem}\label{6c}
Let $\pi\in\rm{Irr}(\sp_{2n}\fq)$ be an irreducible cuspidal representation. Then there exists a character $\psi$ of $\Fq$ such that
\[
\ell^{\rm{FJ}}_{0,\psi}(\pi)\ge\ell^{\rm{FJ}}_{0,\psi'}(\pi)
\]
for any character $\psi'$ of $\Fq$. Moreover, the descent $\cal{D}^{\rm{FJ}}_{\ell_0,\psi}(\pi)$ with respect to $\psi$ is an irreducible cuspidal representation.
\end{theorem}

\begin{proof}

Let $\pi'\in \rm{Irr}(\sp_{2(n-\ell_0)}\fq)$ be an irreducible representation appearing in $\CD^\rm{FJ}_{\ell_0, \psi}(\pi)$. By the definition of descent, we have
 \[
 m_\psi(\pi,\pi')\ne 0
 \]
 and
 \[
 m_\psi(\pi,\pi'')= 0
 \]
 for any  $\pi''\in \rm{Irr}(\sp_{2(n-\ell')}\fq)$ with $\ell'>\ell_0$.

Following the notations in Theorem \ref{u1}, we write $\pi=\prll$ and $\pi'=\prllc$. Since the Lusztig correspondence sends cuspidal representation to cuspidal, these marks $\rho$, $\Lambda$ and $\Lambda'$ correspond to cuspidal representation via (\ref{lus3}).
 Then we can get the explicitly description of symbols $\Lambda_1$ and $\Lambda_1'$ of $\pi'$ by Theorem \ref{u1} and similar argument in subsection \ref{sec6.2}. Actually, because of the cuspidality of $\pi$, some parts of symbols $\Lambda$ and $\Lambda'$ are empty (the explicitly description of symbols for unipotent cuspidal representations can be found in \cite[subsection 3.2, 3.3]{P3}). So the calculation and description of $\Lambda_1$ and $\Lambda_1'$ will be easier than it is in subsection \ref{sec6.2} and will be left to readers. In fact, the symbols $\Lambda_1$ and $\Lambda_1'$ also correspond to unipotent cuspidal representations. With the same argument is subsection \ref{sec6.1}, we would get the explicitly description of $\rho'$.  In particular, we find out that $\pi'=\prllc$ is also cuspidal.  By \cite[Theorem 1.1 (i), Corollary 1.3 (i)]{Wang1}, we can conclude that for any $\pi^*\in \rm{Irr}(\sp_{2(n-\ell_0)}\fq)$,  $m_\psi(\pi,\pi^*)\ne 0$ if and only if $\pi^*=\pi'$. Moreover, can deduce the multiplicity $m_\psi(\pi,\pi')$ to a multiplicity $m_\psi(\pi_\rho,\pi_\rho')$ where $\pi_\rho$ and $\pi'_\rho$ are irreducible representation in Case (A), and $\pi'_\rho$ is the descent of $\pi_\rho$. It follows from Theorem \ref{d1} that $m_\psi(\pi,\pi')=1$, i.e. the descent $\CD^\rm{FJ}_{\ell_1, \psi}(\pi)$ is an irreducible cuspidal representation.

\end{proof}

\section{Lemma in \cite{GRS} and exchanging roots Lemma}\label{sec7}
In this section, we recall the result for cuspidal representaiton in \cite{GRS}. And we generalize their results from cuspidal representations to the representations considered in this paper.

\subsection{Lemma in \cite{GRS}}
Although the lemmas in \cite{GRS} is in the global case, their proof also works over finite fields. Moreover, in \cite{GRS2}, they generalize their results and prove the ``exchanging roots Lemma" which also holds over finite fields. We list the finite field version of their results in this section.

Let $G = \sp_{2n}$. Let $U_r$ be the unipotent radical of the parabolic subgroup of $G$ whose Levi part
is $\GGL_1^r \times \sp_{2(n-r)}$. We define a character $\psi_{U_r}$ of $U_r$ as follows. If $u\in U_r^F$ then
$\psi_{U_r}:=\psi(u_{1,2}+u_{2,3}+\cdots+u_{r,r+1})$ where $\psi$ is a character of $\Fq$.

\begin{lemma}[finite fields case of Lemma 2.2 in \cite{GRS}]\label{g1}
Let $\pi$ be an irreducible cuspidal representation of $\sp_{2n}\fq$.
If
\[
\langle\pi,\psi_{U_r}\rangle_{U_r^F}=\frac{1}{U_r^F}\sum_{u\in U_r^F}\psi_{U_r}(u^{-1})\pi(u)\ne 0,
\]
 then there exists a number
$n \ge m > r$ such that $\pi$ has a nontrivial Fourier coefficient on $\wco_{((2m),1^{2(k-m)})}$.
\end{lemma}

\begin{lemma}[finite fields case of Lemma 2.4 in \cite{GRS}]\label{g2}
Let $\pi$ be an irreducible cuspidal representation of $\sp_{2n}\fq$.
Let $\co_\Gamma\subset \wco_{((2r+1)^2,d_2,d_3,\cdots)}$ and $2r+1\ge d_i$ for all $i$, and suppose
that $\cal{O}_\Gamma$ supports $\pi$. Then there exists a number $m$ such that $2m > 2r + 1$ and
$\co_{\Gamma'}$ supports $\pi$ with $\co_{\Gamma'}\subset\wco_{(2m,1^{2(k-m)})}$.
\end{lemma}

The cuspidality assumption of
$\pi$ in \cite{GRS} was used to obtain Lemma \ref{g1}. To prove Lemma \ref{g2}, one needs Lemma \ref{g1}. Moreover, Lemma \ref{g1} is the only part we need cuspidality. In a word, if one can generalize Lemma \ref{g1} from cuspidal representations to general case, then every Lemma above holds for arbitrary representation. We extend Lemma \ref{g1} as follow:

\begin{lemma}\label{ex1}
Let $\pi$ be an irreducible representation of $\sp_{2n}\fq$. Let $P$ be a parabolic subgroup of $\sp_{2n}$ with Levi factor $\GGL_{r+1}\times\sp_{2(n-r-1)}$. Assume that
\[
\langle\pi,\rm{Ind}^{G^F}_{P^F}(\sigma\otimes\tau)\rangle= 0
\]
for any generic representation $\sigma$ of $\GGL_{r+1}\fq$ and any irreducible representation $\tau$ of $\sp_{2(n-r-1)}\fq$.
If
\[
\langle\pi,\psi_{U_r}\rangle_{U_r^F}=\sum_{u\in U_r^F}\psi_{U_r}(u^{-1})\pi(u)\ne 0,
\]
for some choice of $\psi$, then there exists a number
$n \ge m > r$ such that $\pi$ has a nontrivial Fourier coefficient on $\wco_{(2m,1^{2(k-m)})}$.
\end{lemma}
\begin{proof}
Let $P=LU$, and let $U_{r+1}$ be the unipotent radical of the parabolic subgroup of $G$ whose Levi part
is $\GGL_1^{r+1} \times \sp_{2(n-r-1)}$. Let $U'_{r+1}$ be the subgroup of $\GGL_{r+1}$ consist of unipotent upper
triangular matrices. We have $U_{r+1}=U'_{r+1}\ltimes U$, and the character $\psi_{U_r}$ can be naturally view as a regular character of $U_{r+1}$.

 In the proof of Lemma \ref{g1}, we can conclude that
 \[
\sum_{u\in U_{r+1}^F}\psi_{U_{r+1}}(u^{-1})\pi(u)=\langle \pi,\rm{Ind}^{G^F}_{U_{r+1}^F}\psi_{U_r}
\rangle.
\]
 The cuspidality assumption of $\pi$ in \cite{GRS} was only used to obtain following statement:
\[
\langle \pi,\rm{Ind}^{G^F}_{U_{r+1}^F}\psi_{U_r}
\rangle=0.
\]
And above equality holds if and only if
\[
\langle \pi,\rm{Ind}^{G^F}_{\GGL_{r+1}\fq\ltimes U^F}\sigma\otimes1
\rangle= 0
\]
for any generic representation $\sigma$ of $\GGL_{r+1}\fq$. In other words, $\pi$ can not be a irreducible component of $\rm{Ind}^{G^F}_{P^F}\sigma\otimes\pi'$ for any generic representation $\sigma$ of $\GGL_{r+1}\fq$ and any irreducible representation $\pi'$ of $\sp_{2(n-r-1)}\fq$, which is our assumpiton.
\end{proof}

Then we can extend Lemma \ref{g2} as follow:
\begin{lemma}\label{ex2}
Let $\pi$ be an irreducible representation of $\sp_{2n}\fq$. Let $P$ be a parabolic subgroup of $\sp_{2n}$ with Levi factor $\GGL_{r+1}\times\sp_{2(n-r-1)}$. Assume that
\[
\langle\pi,\rm{Ind}^{G^F}_{P^F}(\sigma\otimes\tau)\rangle= 0
\]
for any generic representation $\sigma$ of $\GGL_{r+1}\fq$ and any irreducible representation $\tau$ of $\sp_{2(n-r-1)}\fq$.
Let $\wco=\wco_{((2r+1)^2,d_2,d_3,\cdots)}$ with $2r+1\ge d_i$ for all $i$, and suppose
that $\wco$ supports $\pi$. Then there exists a number $m$ such that $2m > 2r + 1$ and that
$\wco_{(2m,1^{2(k-m)})}$ supports $\pi$.
\end{lemma}
\begin{proof}
 By the proof of Lemma \ref{g2} in \cite{GRS}, we conclude that
\[
\sum_{u\in U_r^F}\psi_{U_r}(u^{-1})\pi(u)\ne 0.
\]
 For more detail, see \cite[p.11]{GRS}. Hence, this lemma follows immediately from Lemma \ref{ex1}

\end{proof}

\subsection{Exchanging roots Lemma}
Let $C\subset G$ be an $F$-stable subgroup of a maximal unipotent subgroup of $G$, and
let $\psi_C$ be a non-trivial character of $C^F$. Assume that there are two unipotent
$F$-stable subgroups $X$, $Y$ , such that the following conditions are satisfied.
 \begin{itemize}
 \item (1)   $X$ and $Y$ normalize $C$;
\item (2)  $X\cap C$ and $Y \cap C$ are normal in $X$ and $Y$ , respectively, and $X \cap C\verb|\|X$ and
$Y \cap C\verb|\|Y $ are abelian;
\item (3) $X^F$ and $Y^F$ preserve $\psi_C$ (when acting by conjugation);
\item (4)  $\psi_C$ is trivial on $(X \cap C)^F$ and on $(Y \cap C)^F$;
\item (5)   $[X, Y ] \subset C$;
\item (6) The pairing $(X \cap C)^F\verb|\|X^F \times(Y \cap C)^F\verb|\|Y^F \to \bb{C}^\times$, given by
\[
(x, y) \to \psi_C([x, y]),
\]
is multiplicative in each coordinate, non-degenerate, and identifies $(Y \cap C)^F\verb|\|Y^F$
with the dual of $(X \cap C)^F\verb|\|X^F$ and $(X \cap C)^F\verb|\|X^F$ with the dual of $(Y \cap C)^F\verb|\|Y^F$.
\end{itemize}
We represent the setup above by the following diagram,
\begin{equation}\label{er}
\begin{matrix}
&&A &&\\
&\nearrow&&\nwarrow&\\
B=CY&&&&D=CX\\
&\nwarrow&&\nearrow&\\
&&C&&
\end{matrix}.
\end{equation}
Here, $A = BX = DY = CXY $. Extend the character $\psi_C$ to a character $\psi_B$, of
$B^F$, and to a character $\psi_D$ of $D^F$, by making it trivial on $Y^F$ and on $X^F$.

\begin{lemma}[Exchanging roots Lemma (Lemma 7.1 in \cite{GRS2})]\label{erl}
Let $\pi$ be an irreducible representation of $G^F$. Then
\[
\langle \pi,\psi_{B}\rangle_{B^F}=\sum_{u\in B^F}\pi(u)\psi_B(u^{-1})=\sum_{y\in (Y \cap C)^F\verb|\|Y^F}\sum_{u\in D^F}\pi(uy)\psi_D(u^{-1})
\]
\end{lemma}

To compose a Fourier coefficient with another Fourier coefficient, we need the following lemma.
The proof of the following lemma is similar to the proof of Lemma 2.6 in \cite{GRS} which is a special case of Exchanging Roots Lemma. But we would like to give a direct proof here.
\begin{lemma}\label{erL3}
Let $\pi$ be an irreducible representation of $\sp_{2n}\fq$.
Let $\cal{O}_\Gamma=\co_{((2n_1,d_2 \cdots d_s),(q_i))}$ where $2n\ge d_i$, for all $i$ and $d_2 + \cdots + d_s +
2n_1 = 2n$. Then
\[
\rm{dim}\cal{F}(\pi,\co_{\Gamma},\psi)=\rm{dim}\cal{F}(\cal{F}(\pi,\co_{(2n_1,1^{2(n-n_1)}),(q_1)},\psi ),\co_{(d_2,d_3,\cdots,d_s),(q_i')},\psi)
\]
where
\[
(q_i')=\left\{
\begin{array}{ll}
(q_2,\cdots),&\textrm{ if }d_2<2n_1;\\
((-1)^{\#\{i|d_i=2n_1\}},q_2,\cdots),&\textrm{ if }d_2=2n_1.\\
\end{array}
\right.
\]
\end{lemma}
\begin{proof}
 In our proof, we still need to pick some unipotent subgroups of $\sp_{2n}$ which are similar to $X,Y,A,B,C,D$ as before. We will only give the proof for $d_2<2n_1$, the $d_2=2n_1$ case is similar and will be left to the reader. In the $d_2=2n_1$ case, we can calculate the first term of $(q_i')$ by (\ref{sgn}).

Let $G_{\ge k}$ (resp. $G_{\ge k}'$, $G_{\ge k}''$) be certain subgroups of unipotent subgroup of $\sp_{2n}$ (resp. $\sp_{2n}$, $\sp_{2(n-n_1)}$) corresponding to $\co_{(2n_1,d_2,d_3,\cdots,d_s),(q_i)}$ (resp. $\co_{(2n_1,1^{2(n-n_1)}),(+)}$, $\co_{(d_2,d_3,\cdots,d_s),(q_i')}$) with $k\in\{1,1.5,2\}$.

To prove this lemma, we only need to show
\begin{equation}\label{ex3}
\frac{1}{|G_{\ge 1.5}^F|}\sum_{u\in G_{\ge 1.5}^F}\pi(u)\psi(u^{-1})=\frac{1}{|G_{\ge 1.5}^{\prime F}||G_{\ge 1.5}^{\prime \prime F}|}\sum_{u_2\in G_{\ge 1.5}^{\prime\prime F}}\left(\sum_{u_1\in G_{\ge 1.5}^{\prime F}}\pi(u_1u_2)\psi(u_1^{-1})\right)\psi(u_2^{-1})
\end{equation}

{\bf Step 1}: pick unipotent subgroups $X,Y,A,B,C,D$ as in Exchanging Roots Lemma.

From the structure of $\co_\Gamma$, we can pick a representative element $\gamma=\{H,X,Y\}$ in the orbit of $\frak{sl}_2$-triple corresponding to $\co_\Gamma$ such that
\[
\rm{exp}tH = \rm{diag}(t^{2n_1-1},  \cdots, t^{d_2-1}, \cdots , t^{-(d_2-1)}, \cdots ,  t^{-2n_1+1})\in \sp_{2n}\fq
\]
where the powers of $t$ descend.
There is a Weyl group element $w$ which conjugates $\rm{exp}tH$ to the torus
\[
\begin{aligned}
h(t):&=w(\rm{exp}tH)w^{-1} = \rm{diag}(T,T_0,T').
\end{aligned}
\]
where
\[
T=(t^{2n_1-1}, t^{2n_1-3}, \cdots ,t),\quad T_0=( t^{d_2-1}, \cdots , t^{-(d_2-1)})\textrm{ and } T'=(t^{-1},\cdots,t^{-(2n_1-3)}, t^{-(2n_1-1)})
\]
with the powers of $t$ descending.
Similarly, from the structure of $\co_{(d_2, \cdots ,d_s),(q_i')}$, we can pick a representative element $\gamma''=\{H'',X'',Y''\}$ in the orbit of $\frak{sl}_2$-triple corresponding to $\co_{(d_2, \cdots ,d_s),(q_i')}$ such that
\[
\rm{exp}tH'' = \rm{diag}(t^{d_2-1},  \cdots , t^{-(d_2-1)})\in \sp_{2(n-n_1)}\fq
\]
where the powers of $t$ goes down. Let $\fg_i = \{Z \in \fg | \rm{Ad}(h(t))(Z) = iZ\}$ and
\[
\fg = \bigoplus_{j\in\bb{Z}}\fg_{j}.
\]
Let
\[
\fg_{\ge i}:=\bigoplus_{i\ge j}\fg_{j}\textrm{ and } \fg_{\le i}:=\bigoplus_{i\le j}\fg_{j}.
\]
They are Lie algebras of a close connect subgroup $G_{\ge i}$ and $G_{\le i}$ of $G$, respectively.

To write down $G_{\ge i}^F$ explicitly, we set
\begin{itemize}
 \item $z_i$ is the number of $t^{j}$ in $\rm{exp}tH''$ such that $j> 2n_1-1-2i$;
\item $z'_j$ is the number of $t^{i}$ in $\rm{exp}tH''$ such that $i\ge 2n_1+3-2j$;
\item  $z^1_i$ is the number of $t^{j}$ in $\rm{exp}tH''$ such that $j= 2n_1-2i$.
\item $z^{\prime 1}_j$ is the number of $t^{i}$ in $\rm{exp}tH''$ such that $i= 2n_1+2-2j$.
\end{itemize}
Note that for symplectic group $\sp_{2n}$, nilpotent orbits are parameterizes by all partitions
of $2n$ where odd numbers occur with even multiplicity. So for a even number $i$, the multiplicity of $t^{i}$ occurring in $\rm{exp}tH''$ is even. In particular, $z^1_i$ and $z^{\prime 1}_j$ are even.

We write
\[
G_{\ge1.5}=\left\{\begin{pmatrix}
u&h+e&g\\
&v&h^*+e^*\\
&&u^*
\end{pmatrix}
\begin{pmatrix}
I&&\\
k+f&I&\\
&k^*+f^*&I
\end{pmatrix}\right\}
\]
where
\begin{itemize}
 \item $h$ runs over $M_{\ge2}:=\{(\alpha_{ij})_{n_i\times2(n-n_1)}|\alpha_{ij}=0\textrm{ if }j\le z_i\}$;
\item $u$ runs over the upper triangle matrixes of $\GGL_{n_1}$;
\item  $g$ runs over $M=\{(g_{ij})_{n_1\times n_1}\}$;
\item $v$ runs over $G_{\ge1.5}''$;
\item $k$ runs over $ M_{\ge2}':=\{(\alpha_{ij})_{ 2(n-n_1)\times n_1}|\alpha_{ij}=0\textrm{ if }i> z'_{j}\}$
\item $e $ runs over $ M_1=\{(\alpha_{ij})_{n_1\times 2(n-n_1)}|\alpha_{ij}=0\textrm{ if }j\le \frac{z^1_i}{2}\textrm{ or }j>z_i\}$;
\item $f$ runs over $ M_1'=\{(\alpha_{ij})_{ 2(n-n_1)\times n_1}|\alpha_{ij}=0\textrm{ if }i> (\frac{z^{\prime 1}_{j}}{2}+z^{\prime }_{j})\textrm{ or }i\le z^{\prime }_{j}\}$.
\end{itemize}
For $\co_{(2n_1,1^{2(n-n_1)})}$, we have
\[
G'_{\ge1.5}=\left\{\begin{pmatrix}
u&h'&g\\
&I&h^{\prime*}\\
&&u^*
\end{pmatrix}\right\}
=
\left\{\begin{pmatrix}
u&h+e&g\\
&I&h^*+e^*\\
&&u^*
\end{pmatrix}
\begin{pmatrix}
I&d&\\
&I&d^{*}\\
&&I
\end{pmatrix}
\right\}
\]
where
\begin{itemize}

\item $h'$ runs over $\{(\alpha_{ij})|\alpha_{n_1j}=0\textrm{ if }j>n+n_1\}$;

\item $d$ runs over $\{(\alpha_{ij})|\alpha_{ij}=0 \textrm{ if }j> \frac{z^1_i}{2}\}$.
 \end{itemize}

We set
\[
C:=\left\{\begin{pmatrix}
u&h+e&g\\
&v&h^*+e^*\\
&&u^*
\end{pmatrix}\right\},\quad X:=\left\{\begin{pmatrix}
I&d&\\
&I&d^{ *}\\
&&I
\end{pmatrix}
\right\}
\textrm{ and }
Y:=\left\{\begin{pmatrix}
I&&\\
k+f&I&\\
&k^*+f^*&I
\end{pmatrix}\right\},
\]
and let $\frak{c}$, $\frak{x}$ and $\frak{y}$ be their Lie algebra, respectively. Let
\[
B:=CY,\ D:=CX\textrm{ and }A:=CXY.
\]
Note that
 $(Y \cap C)\verb|\|Y=I\textrm{ and }(X \cap C)\verb|\|X=I$.

{\bf Step 2}: prove (\ref{ex3}).

To prove it, we would like to show that both side of (\ref{ex3}) is equal to
\[
\frac{1}{| C^F||X^F||Y^F|}\sum_{u\in C^F}\sum_{x\in X^F}\sum_{y\in Y^F}\pi(uyx)\psi(u^{-1}).
\]
 We set unipotent subgroups $X_r$ and $Y_r$ and their Lie algebra $\frak{x}_r$ and $\frak{y}_r$ as follow:
\[
X_r=X\cap (G_{\le2-r}\verb|\|G_{<2-r})\textrm { and }
Y_r=Y\cap (G_{\ge r}\verb|\|G_{> r}).
\]
In other words,
\[
\frak{x}_r=\fg_{2-r}\cap\frak{x} \textrm{ and }\frak{y}_r=\fg_{r}\cap\frak{y}.
\]
Note that $X$ and $Y$ are abelian, we have
\[
Y=\prod_{r=1}^{m}Y_r
\textrm{ and }
X=\prod_{r=1}^{m}X_r.
\]
To prove (\ref{ex3}), our strategy is to write the Fourier expansion of both side of (\ref{ex3}) on $X_r^F$ and $Y_r^F$ for $1\le r\le m$, respectively.

We now turn to describe elements in $\frak{x}_r$ and $\frak{y}_r$ explicitly.
Let
$\beta_r$ be the power of $t$ in the $r$-th term of $w(\rm{exp}tH)w^{-1}$ and $\delta(r,i)=\frac{2n_1-\beta_i+r-1}{2}$. Let
\[
\cal{X}_r:=\{(i,j)|e_{ij}\in \mathfrak{x}_r\}\textrm{ and }\cal{Y}_r:=\{(i,j)|e_{ij}\in \frak{y}_r\}.
\]
To be more precise, we have
\[
\cal{X}_r=\{(i,j)|\ n_1< j\le n,\ i=\delta(r,j)\textrm{ and }i<n_1\}
\]
and
\[
\cal{Y}_r=\{(i,j)|\ n_1< i\le n,\ j=\delta(r,i)+1\textrm{ and }j\le n_1\}.
\]
Then $x\in X_r $ (resp. $y\in Y_r $) if and only if
\[
x_r=I+\sum_{(i,j)\in\cal{X}_r} \alpha_{i,j}e_{ij} \ \textrm{ (resp. } y_r=I+\sum_{(i,j)\in\cal{Y}_r} \alpha_{i,j}'e_{ij}\textrm{)}
\]
for some $\{\alpha_{i,j}\}$ (resp. $\{\alpha'_{i,j}\}$).
By direct calculation, we have $[X_r,Y_r]\subset G_{> 2}\verb|\|G_{\ge 2}$, and $[X_r,Y_{r'}]\subset G_{> 2}$ for $r'>r$. So
$
\psi([x_r,y_r'])=1
$
and
\[
\psi([x_r,y_r])=\psi(\sum_{j=n_1+1}^{n}\alpha_{\delta(r,j),j}\cdot\alpha'_{j,\delta(r,j)+1}).
\]
Therefore the pairing $I\verb|\|X_r^F \times I\verb|\|Y_r^F \to \bb{C}^\times$, given by
\[
(x_r, y_r) \to \psi([x_r, y_r]),
\]
is multiplicative in each coordinate, non-degenerate, and identifies $I\verb|\|Y_r^F$
with the dual of $I\verb|\|X_r^F$ and $I\verb|\|X_r^F$ with the dual of $I\verb|\|Y_r^F$.

Since $G_{\ge 1.5}^{\prime\prime }\cdot G_{\ge 1.5}^{\prime }=CX$, we set
\[
\begin{aligned}
\phi_{m+1}(g):=&\frac{1}{|G_{\ge 1.5}^{\prime F}||G_{\ge 1.5}^{\prime\prime F}|}\sum_{u_2\in G_{\ge 1.5}^{\prime\prime F}}\left(\sum_{u_1\in G_{\ge 1.5}^{\prime F}}\pi(u_1u_2g)\psi(u_1^{-1})\right)\psi(u_2^{-1})\\
=&\frac{1}{|C^F||X^F|}\sum_{u\in C^F}\sum_{x\in X^F}\pi(uxg)\psi(u^{-1})
\end{aligned}
\]
and inductively we define
\[
\phi_{r}(g):=\frac{1}{|Y_r^F|}\sum_{y_{r}\in Y_{r}^F}\phi_{r+1}(y_{r}g).
\]
In particular,
\[
\phi_{1}(1)=\frac{1}{|C^F||X^F||Y^F|}\sum_{u\in C^F}\sum_{x\in X^F}\sum_{y\in Y^F}\pi(uxy)\psi(u^{-1}).
\]
We turn to prove $\frac{1}{|Y_{m'}^F|}\phi_{m'+1}(1)=\phi_{m'}(1)$ for $m'\le m$. Let us write the Fourier expansion of $\phi_{m'+1}$ and evaluate it at 1:
\[
\begin{aligned}
\phi_{m'+1}(1)=&\frac{1}{|Y_{m'}^F|}\sum_{\psi^*\in \widehat{Y_{m'}^F}}\sum_{y_{m'}\in Y_{m'}^F}\phi_{m'+1}(y_{m'})\psi^*(y_{m'}^{-1})\\
=&\frac{1}{Y(m')}\sum_{\psi^*\in \widehat{Y_{m'}^F}}\sum_{y_{m'}\in Y_{m'}^F}\sum_{\mbox{\tiny$\begin{array}{c}y_r\in Y_r^F \\
m'+1\le r\le m\\
\end{array}$}}\sum_{u\in C^F}\sum_{x\in X^F}\pi(uxy_my_{m-1}\cdots y_{m'+1}y_{m'})\psi(u^{-1})\psi^*(y_{m'}^{-1})
\end{aligned}
\]
where $Y(m'):=\prod_{r=m'}^m |Y_r^F|$.
To abbreviate notations, we write $y_{m'}'$ instead of $y_my_{m-1}\cdots y_{m'+1}$. Using that fact that one can identify $I\verb|\|X_r^F$ with the dual of $I\verb|\|Y_r^F$, for a fixed $\psi^*$, there exist a $x_{m'}\in X_{m'}^F$ such that $\psi^*(y_{m'})=\psi([y_{m'}^{-1},x_{m'}^{-1}])$ for any $y_{m'}\in Y_{m'}^F$. We get
\[
\begin{aligned}
\pi(uxy_{m'}'y_{m'})\psi(u^{-1})\psi^*(y_{m'}^{-1})=\pi(uxy_{m'}'y_{m'})\psi(u^{-1})\psi([x_{m'}^{-1},y_{m'}^{-1}])\\
\end{aligned}
\]
where $x_{m'}^{-1}$ depends on $\psi^*\in \widehat{Y_{m'}^F}$.
Since $\pi$ is a class function and $XC=CX$, consider the sum over $X^F$ and $C^F$, we get
\[
\begin{aligned}
&\sum_{u\in C^F}\sum_{x\in X^F}\pi(uxy_{m'}'y_{m'})\psi(u^{-1})\psi([x_{m'}^{-1},y_{m'}^{-1}])\\
=&\sum_{u\in C^F}\sum_{x\in X^F}\pi(x_{m'}^{-1}uxy_{m'}'y_{m'}x_{m'})\psi(u^{-1})\psi([x_{m'}^{-1},y_{m'}^{-1}])\\
=&\sum_{u\in C^F}\sum_{x\in X^F}\pi(ux_{m'}^{-1}xy_{m'}'y_{m'}x_{m'})\psi(u^{-1})\psi([x_{m'}^{-1},y_{m'}^{-1}])\\
\end{aligned}
\]
Change variable in the sum over $X^F$, $x\to x_{m'}^{-1}x$, then above equation become
\[
\begin{aligned}
&\sum_{u\in C^F}\sum_{x\in X^F}\pi(uxy_{m'}'y_{m'}x_{m'})\psi(u^{-1})\psi([x_{m'}^{-1},y_{m'}^{-1}])\\
\end{aligned}
\]
Now we want to move $x_{m'}$ through $y_{m'}'$ and $y_{m'}$. First, for $y_{m'}$, one has
\[
\begin{aligned}
&\pi(uxy_{m'}'y_{m'}x_{m'})\psi(u^{-1})\psi([x_{m'}^{-1},y_{m'}^{-1}])\\
=&\pi(uxy_{m'}'x_{m'}y_{m'}[y^{-1}_{m'},x^{-1}_{m'}])\psi(u^{-1})\psi([x_{m'}^{-1},y_{m'}^{-1}])\\
=&\pi([y^{-1}_{m'},x^{-1}_{m'}]uxy_{m'}'x_{m'}y_{m'})\psi(u^{-1})\psi([x_{m'}^{-1},y_{m'}^{-1}])\\
\end{aligned}
\]
Change variable in the sum over $C^F$, $u\to [y^{-1}_{m'},x^{-1}_{m'}]u$, we have
\[
\begin{aligned}
&\sum_{u\in C^F}\sum_{x\in X^F}\pi(uxy_{m'}'x_{m'}y_{m'})\psi(u^{-1}[y^{-1}_{m'},x^{-1}_{m'}])\psi([x_{m'}^{-1},y_{m'}^{-1}])=\sum_{u\in C^F}\sum_{x\in X^F}\pi(uxy_{m'}'x_{m'}y_{m'})\psi(u^{-1})\\
\end{aligned}
\]
For $y_{m'}'$, one has
\[
\begin{aligned}
\sum_{y_{m'}\in Y_{m'}^F}\pi(uxy_{m'}'x_{m'}y_{m'})\psi(u^{-1})=\sum_{y_{m'}\in Y_{m'}^F}\pi(uxu'x_{m'}y_{m'}'y_{m'})\psi(u^{-1})\\
\end{aligned}
\]
Note that $u'=[y_{m'}',x_{m'}]\in G^F_{>2}\subset C^F$ and $\psi$ is trivial on $G^F_{>2}$, by changing $u\to uu'$, we get
\[
\begin{aligned}
&\sum_{u\in C^F}\sum_{x\in X^F}\pi(uxu'x_{m'}y_{m'}'y_{m'})\psi(u^{-1})\\
=&\sum_{u\in C^F}\sum_{x\in X^F}\pi(uu'xx_{m'}y_{m'}'y_{m'})\psi(u^{-1})\\
=&\sum_{u\in C^F}\sum_{x\in X^F}\pi(uxx_{m'}y_{m'}'y_{m'})\psi(u^{-1}).
\end{aligned}
\]
Changing $x$ to $xx_{m'}$, above equation become
$
\sum_{u\in C^F}\sum_{x\in X^F}\pi(uxy_{m'}'y_{m'})\psi(u^{-1})
$.
Thus
\[
\begin{aligned}
\phi_{m'+1}(1)=&\frac{1}{|Y_{m'}^F|}\sum_{\psi^*\in \widehat{Y_{m'}^F}}\sum_{y_{m'}\in Y_{m'}^F}\phi_{m'+1}(y_{m'})\psi^*(y_{m'}^{-1})\\
=&\frac{1}{Y(m')}\sum_{\psi^*\in \widehat{Y_{m'}^F}}\sum_{y_{m'}\in Y_{m'}^F}\sum_{\mbox{\tiny$\begin{array}{c}y_r\in Y_r^F \\
m'+1\le r\le m\\
\end{array}$}}\sum_{u\in C^F}\sum_{x\in X^F}\pi(uxy_my_{m-1}\cdots y_{m'+1}y_{m'})\psi(u^{-1})\\
=&\frac{1}{Y(m'+1)}\sum_{\mbox{\tiny$\begin{array}{c}y_r\in Y_r^F \\
m'\le r\le m\\
\end{array}$}}\sum_{u\in C^F}\sum_{x\in X^F}\pi(uxy_my_{m-1}\cdots y_{m'+1}y_{m'})\psi(u^{-1})\\
=&{|Y_{m'}^F|}\phi_{m'}(1)
\end{aligned}
\]
Hence $\phi_{m+1}(1)={|Y_{m}^F|}\phi_{m}(1)={|Y_{m}^F||Y_{m-1}^F|}\phi_{m-1}(1)=\cdots={|Y^F|}\phi_{1}(1)$, and we have
\begin{equation}\label{ex4}
\begin{aligned}
&\frac{1}{|G_{\ge 1.5}^{\prime F}||G_{\ge 1.5}^{\prime\prime F}|}\sum_{u_2\in G_{\ge 1.5}^{\prime\prime F}}\left(\sum_{u_1\in G_{\ge 1.5}^{\prime F}}\pi(u_1u_2)\psi(u_1^{-1})\right)\psi(u_2^{-1})\\
=&\phi_{m+1}(1)\\
=&\frac{1}{|Y^F|}\phi_{1}(1)\\
=&\frac{1}{|C^F||X^F||Y^F|}\sum_{u\in C^F}\sum_{x\in X^F}\sum_{y\in Y^F}\pi(uxy)\psi(u^{-1})
\end{aligned}
\end{equation}
By exchanging the role of $X$ and $Y$, with same argument, we have
\begin{equation}\label{ex5}
\frac{1}{|G_{\ge 1.5}^F|}\sum_{u\in G_{\ge 1.5}^F}\pi(u)\psi(u^{-1})=\frac{1}{|C^F||X^F||Y^F|}\sum_{u\in C^F}\sum_{x\in X^F}\sum_{y\in Y^F}\pi(uyx)\psi(u^{-1})
\end{equation}
Note that $\pi$ is a class function and $XC=CX$,
\begin{equation}\label{ex6}
\sum_{u\in C^F}\sum_{x\in X^F}\sum_{y\in Y^F}\pi(uyx)\psi(u^{-1})=\sum_{u\in C^F}\sum_{x\in X^F}\sum_{y\in Y^F}\pi(xuy)\psi(u^{-1})=\sum_{u\in C^F}\sum_{x\in X^F}\sum_{y\in Y^F}\pi(uxy)\psi(u^{-1}).
\end{equation}
Then (\ref{ex3}) follows immediately from (\ref{ex4}), (\ref{ex5}) and (\ref{ex6}).
\end{proof}
\section{Proof of main results}\label{sec8}
\subsection{Wavefront sets of representations in Case (A)}
Let $G$ be a symplectic group over $\Fq$, and let $\pi\in \cal{E}(G^F,s)$. In this section, we consider the Case (A) in Section \ref{sec6}, i.e. we assume that $s$ has no eigenvalues $\pm1$. The goal of this section is to prove that we can get the wavefront set of $\pi$ by the descent method. Recall that for an irreducible representation $\pi$, we can
associate it with an array $\hat{\ell}{(\pi)}$ in Section \ref{sec5.4}. Furthermore, by Corollary \ref{da}, $\hat{\ell}{(\pi)}$ is a partition if $\pi$ is in Case (A).

\begin{theorem}\label{am} Let $G$ be a symplectic group over $\Fq$, and let $\pi\in \cal{E}(G^F,s)$. Assume that $s$ has no eigenvalues $\pm1$. Let $\wco$ be the $F$-stable nilpotent orbit corresponding to $\hat{\ell}{}{}(\pi)$. Then for every $F$-rational nilpotent orbit $\co_\Gamma\subset \wco$, we have
\[
\rm{dim}\rm{Hom}_{G_{\ge 1}^F}(\pi,\omega_{\psi })=1
\]
where $G_{\ge 1}$ is certain unipotent subgroup of $G$ corresponding to $\co_\Gamma$.
Moreover, if $C_{G^{*F}}(s)$ is a product of general linear groups, then $\wco$ is the wavefront set of $\pi$.
\end{theorem}
\begin{proof}
We prove the proposition by induction on $n$.

In Section \ref{sec6.1}, we know that the descent of $\pi$ does not depend on the choice of character $\psi$. By (\ref{lfd}), (\ref{bd}), (\ref{5.3}) and Theorem \ref{d1}, for every $F$-rational orbit
\[
\co_{\Gamma_1}\subset\wco_{(\hat{\ell}{}{(\pi)}_1,1^{2n-\hat{\ell}{}{(\pi)}_1})}=\wco_{(2\ell_0^{\rm{FJ}}(\pi),1^{2n-2\ell_0^{\rm{FJ}}(\pi)})},
 \]
 $\pi$ has a nonzero Fourier
coefficient with respect to any choices of character $\psi$
and $\pi$ has no nontrivial Fourier
coefficient on $\forall\ \co_{\Gamma'}\subset\wco_{(\hat{\ell}{}{(\pi)}_1+2i,1^{2n-\hat{\ell}{}{(\pi)}_1-2i})}$ with $i>0$.

Assume that $C_{G^{*F}}(s)$ is a product of general linear groups. If $\pi$ has a nonzero Fourier
coefficient on certain $F$-rational nilpotent orbit in the $F$-stable nilpotent orbit $\wco_{((\hat{\ell}{}{(\pi)}_1+2i+1)^2,d_1,d_2,\cdots)}$ with $i\ge 0$, then by Proposition \ref{prop6.6} and Lemma \ref{ex2}, we conclude that $\pi$ has a nonzero Fourier
coefficient on certain $F$-rational nilpotent orbit in $\wco_{(\hat{\ell}{}{(\pi)}_1+2i+2,1^{2n-\hat{\ell}{}{(\pi)}_1-2i-2})}$, which is impossible. So $\pi$ has no nontrivial Fourier
coefficient on any $F$-rational nilpotent orbits in $\wco_{((\hat{\ell}{}{(\pi)}_1+2i+1)^2,1^{2n-2\hat{\ell}{}{(\pi)}_1-4i-2})}$. Hence the first part of the wavefront set of $\pi$ have to be $\hat{\ell}{}{(\pi)}_1=2\ell_0^{\rm{FJ}}(\pi)$.
 By Lemma \ref{erL3},
 \[
 \co_{\Gamma}=\co_{((\hat{\ell}{}{(\pi)}_1,d_2,d_3,\cdots),(q_i))}\subset\wco_{(\hat{\ell}{}{(\pi)}_1,d_2,d_3,\cdots)}
  \]
  supports $\pi$ if and only if
  \[
  \co_{((\hat{\ell}{}{(\pi)}_1,1^{2n-2\lambda[a]_1}),(q_1))}\circ\co_{((d_2,d_3,\cdots),(q_i'))}
   \]
   supports $\pi$. In other words, we need check whether $\co_{((d_2,d_3,\cdots),(q_i'))}$ supports $\CD^\rm{FJ}_{\psi}(\pi)$. It is easy to see that
   \[
   \wco_{((\hat{\ell}{}{(\pi)}_1),d_2,d_3,\cdots)}\ge \wco_{((\hat{\ell}{}{(\pi)}_1),d_2',d_3',\cdots)}
   \]
     if and only if
     \[
     \wco_{(d_2,d_3,\cdots)}\ge\wco_{(d_2',d_3',\cdots)}.
     \]
So if $\wco_{((\hat{\ell}{}{(\pi)}_1),d_2,d_3,\cdots)}$
      is the wavefront set of $\pi$, then $\wco_{(d_2,d_3,\cdots)}$ is the wavefront set of $\CD^\rm{FJ}_{\psi}(\pi)$. According to Theorem \ref{d1}, in Case (A), $\CD^\rm{FJ}_{\psi}(\pi)$ is also an irreducible representation in Case (A). Moreover, assume that $\CD^\rm{FJ}_{\psi}(\pi)=\pi'\in\cal{E}(G^{\prime F},s')$, the centralizer $C_{G^{\prime *F}}(s')$ is also a product of general linear groups. By induction hypothesis, the wavefront set of $\CD^\rm{FJ}_{\psi}(\pi)$ is corresponding to $\hat{\ell}{}{(\CD^\rm{FJ}_{\psi}(\pi))}=(\hat{\ell}{}{(\pi)}_2,\hat{\ell}{}{(\pi)}_3,\cdots)$, which completes the proof. The multiplicity 1 result immediately follows from Lemma \ref{erL3} and our induction hypothesis.

     Assume that $C_{G^{*F}}(s)$ is not a product of general linear groups. Applying Lemma \ref{erL3} and with the same argument of composing nilpotent orbits, we can still get
    $
\rm{dim}\rm{Hom}_{G_{\ge 1}^F}(\pi,\omega_{\psi })=1,
$
     although one can not conclude that $\wco$ is the maximal orbit support $\pi$.
\end{proof}

\subsection{Wavefront sets of representations in Case (B)}
Assuming that the characteristic $p$ of $\Fq$ is suffice large, and $G$ is a split connected reductive group $G$ defined over $\Fq$.
For every irreducible character $\pi$ of $G^F $, there is
a unique unipotent class $\co$ in $G$ which has the property:
 \begin{itemize}
 \item[] (1)  $\sum_{u\in\co^F}\pi(u)\ne 0$;
\item[] (2) $\sum_{u\in\co^{\prime F}}\pi(u)\ne 0$ implies $\rm{dim}\co^{\prime F}\le\rm{dim}\co^{ F}$ with $\co'$ any $F$-stable unipotent
class.
\end{itemize}
The class $\co$ is called the unipotent support of $\pi$ (see \cite{L6}, \cite{GM}, \cite{AA}).
In particular, for an unipotent representation $\pl$ of $\sp_{2n}\fq$, we can get unipotent support $\co$ of $\pl$ as follow (see \cite[Section 4]{AKP}):
 \begin{itemize}\label{us}
 \item Pick a special symbol $\cal{Z}$
which contains the same entries with
the same multiplicities as symbol $\Lambda$.
\item Pick a bipartition $(\eta,\zeta)=(\eta_1,\cdots,\eta_k,\zeta_1,\cdots,\zeta_r)$ by $\Upsilon(\cal{Z})$. We
ensure that $k=r=m$ have by removing zeros as parts where necessary.
\item Pick the unique partition $\lambda$ such that $\psi(\lambda)=(\eta,\zeta)$ where $\psi$ is defined in (\ref{p1});
\item Pick the unipotent orbit $\co$ corresponding to partition $\lambda$.
\end{itemize}

Lusztig has proved in \cite[Theorem 11.2]{L7}, under the assumption that $p$
is large enough, that the closure of the unipotent support of $\pi$ coincides with
the wavefront set defined by Kawanaka of its Alvis-Curtis dual. In summary, we have the following theorem.

\begin{theorem}\label{u9}
Assuming that the characteristic $p$ of $\Fq$ is suffice large. Let $\pl\in\rm{Irr}(\sp_{2n}\fq)$ and $\lambda$ be a partition defined as above. Then the wavefront set of $\pi$ is $\co_{\lambda^t}$.
\end{theorem}

As in the Section \ref{sec6.2}, in this section, we write $\cal{D}^{\rm{FJ}}$ instead of $\cal{D}^{\rm{FJ}}_\psi$.

\begin{theorem}
Let $\pl\in\rm{Irr}(\sp_{2n}\fq)$ and $\widetilde{\ell}{}$ be the descent partition of $\pi$. Then $\co_{\widetilde{\ell}{}}$ is the wavefront set of $\pl$.
\end{theorem}
\begin{proof}
We only give the proof for unipotent representation $\pl$. Since the descent of an irreducible unipotent representation is a $\theta$-representation and every irreducible $\theta$-representation coincides with the descent of certain irreducible unipotent representations, we can deduce the $\theta$-representation case from the unipotent case.

Let $\lambda^t$ be the partition given in Theorem \ref{u9} corresponding to $\pl$ and $\hat{\ell}{(\pl)}=(\ell_1,\ell_2,\cdots)$ be the largest descent index.
By Proposition \ref{u5.5}, there are two possibilities:
 \begin{itemize}\label{us}
 \item[] Case (1): $\ell_1\ge\ell_2$. In this case, by Proposition \ref{u7}, we have $\ell_1\ge\ell_i$ for $i> 1$. Moreover, $\cal{D}^{\rm{FJ}}(\pl)$ is irreducible;
 \item[] Case (2): $\ell_1=\ell_2-1$.
\end{itemize}
We only prove the $\rm{def}(\Lambda)>0$ case, the proof of the $\rm{def}(\Lambda)<0$ case is similar.
Assume that $\rm{def}(\Lambda)=2k+1$ and
\[
\Lambda=\begin{pmatrix}
a_1,a_2,\cdots,a_{m+2k+1}\\
b_1,b_2,\cdots,b_{m}
\end{pmatrix}.
\]
Let $\cal{Z}$ be the special symbol containing the same entries with
the same multiplicities as symbol $\Lambda$. We have
\[
\cal{Z}=\begin{pmatrix}
a'_1,&a'_2,&\cdots,&a'_{m+k},&a'_{m+k+1}\\
b'_1,&b'_2,&\cdots,&b'_{m+k}&
\end{pmatrix}
\]
with $a'_{m+k}>0$, $b'_{m+k}>0$. By Proposition \ref{u5.5} (ii), we conclude that
$a'_{m+k+1}=0 $ if $\pl$ is in Case (1) and $a'_{m+k+1}>0 $ otherwise.

 Consider Case (1). Let $(\eta,\zeta)=(\eta_1,\cdots,\eta_h,\zeta_1,\cdots,\zeta_r)$ be the bipartition given by $\Upsilon(\cal{Z})$. We ensure that $h = r=k+m$ by adding zeros as parts where necessary. Note that $\zeta_r=b'_{m+k}> 0$. Hence $\psi^{-1}(\eta,\zeta)=\lambda=(\lambda_1,\cdots,\lambda_{2(k+m)})$. Therefore, the first part of $\lambda^t$ and  wavefront set is $2(k+m)$. On the other hand, by theorem \ref{u5}, we have $\ell_1=\left(\Upsilon(\Lambda\right)_{*})^t_1+k=m+k$. So the first part of the partition corresponding the wavefront set of $\pl$ coincides with $\widetilde{\ell}{}_1$.

We have already known the wavefront set of $\pl$ is $\co_{(2(k+m),\mu_1,\mu_2,\cdots)}=\co_{(2(k+m),\mu)}$. By Lemma \ref{erL3}, $\co_{(2(k+m),\mu)}$ is the the wavefront set of $\pl$ if and only if $\mu$ is the biggest partition such that there exist $F$-rational nilpotent orbits $\co_{(\mu,(q_i'))}$ and $\co_{((2(k+m),1,1,\cdots,1),\epsilon)}$ such that
 \[
 \cal{F}(\cal{F}(\co_{((2(k+m),1,1,\cdots,1),\epsilon)},\psi ),\co_{(\mu,(q_i'))},\psi)\ne 0.
 \]
 In other words, $\mu$ is corresponding to the wavefront set of the irreducible representation $\cal{D}^{\rm{FJ}}(\pl)\in\rm{Irr}(\sp_{2(n-m-k)}\fq)$. By induction on $n$, we conclude that $\co_{(\widetilde{\ell}{}_2,\widetilde{\ell}{}_3,\widetilde{\ell}{}_4,\cdots)}$ is the wavefront set of $\cal{D}^{\rm{FJ}}(\pl)$, which completes the proof of the Case (1).

For Case (2), note that $a'_{m+k+1}>0 $. By Proposition \ref{u5.5} (ii), we have $a_{m+2k+1}>0$ and $b_{m}>0$. With similar process, we get that $\lambda=(\lambda_1,\cdots,\lambda_{2(k+m)+1})$ and $\lambda_{2(k+m+1)}\ge 2$. Then the first two parts of partition $\lambda^t$ is $((2(k+m)+1)^2)$. We write $\lambda^t=((2(k+m)+1)^2,\lambda_1^{* t}, \lambda_2^{* t},\cdots)=((2(k+m)+1)^2,\lambda^{* t})$. Consider the irreducible representation $\pi_{\Lambda^*}$ with
\[
\Lambda^*=\begin{pmatrix}
a_1-1,a_2-1,\cdots,a_{m+2k+1}-1\\
b_1-1,b_2-1,\cdots,b_{m}-1
\end{pmatrix}.
\]
Assume that $a_{m+2k+1}-1$ and $b_{m}-1$ are not both 0.  Then the corresponding special symbol is
\[
\cal{Z}^*=\begin{pmatrix}
a'_1-1,&a'_2-1,&\cdots,&a'_{m+k}-1,&a'_{m+k+1}-1\\
b'_1-1,&b'_2-1,&\cdots,&b'_{m+k}-1&
\end{pmatrix}
\]
and $(\eta^*,\zeta^*)=\Upsilon(\cal{Z})^*=(\eta_1-1,\cdots,\eta_h-1,\zeta_1-1,\cdots,\zeta_r-1)$.
By direct calculation, the wavefront set of $\pi_{\Lambda^*}$ is corresponding to $\lambda^{* t}$.

If $a_{m+2k+1}-1$ and $b_{m}-1$ are both 0, the we consider
\[
\Lambda^*=\begin{pmatrix}
a_1-1,a_2-1,\cdots,a_{m+2k}-1,0\\
b_1-1,b_2-1,\cdots,b_{m-1}-1,0
\end{pmatrix}\sim\begin{pmatrix}
a_1-2,a_2-2,\cdots,a_{m+2k}-2\\
b_1-2,b_2-2,\cdots,b_{m-1}-2
\end{pmatrix}.
\]
The wavefront set of $\pi_{\Lambda^*}$ is also corresponding to $\lambda^{* t}$.

On the descent side, by Theorem \ref{u5} and Proposition \ref{u5.5} (ii), there is a descent sequence of $\pl$ as follow:
\[
\pl=\pi_{\Lambda_1}\xrightarrow{\ell_1}\pi_{\Lambda_2}\xrightarrow{\ell_2=\ell_1+1}\pi_{\Lambda_3}\xrightarrow{\ell_3}\pi_{\Lambda_4}\cdots
\]
with $\ell_1=k+m$ and $\Lambda_3=\Lambda^*$. By induction on $n$, we know that the wavefront set of $\pi_{\Lambda_3}\in\rm{Irr}(\sp_{2(n-m-k)}\fq)$ is corresponding to $\lambda^{* t}=(\widetilde{\ell}{}_3,\widetilde{\ell}{}_4,\cdots)$. Note that $\widetilde{\ell}{}_1=\widetilde{\ell}{}_2=\ell_1+\ell_2=2m+2k+1$ and $\lambda^t=((2(k+m)+1)^2,\lambda^{* t})$. Then $\co_{\widetilde{\ell}{}}$ is the wavefront set of $\pl$.
\end{proof}

In the proof of Theorem \ref{am}, we get the multiplicity 1 result by using the fact that the descent of an irreducible representation in Case (A) is also an irreducible representation in Case (A). With similar argument, we get following multiplicity 1 result for Case (B):

\begin{theorem}\label{am2} Let $G=\sp_{2n}$ be a symplectic group over $\Fq$, and let $\pl$ be an irreducible unipotent (resp, $\theta$-) representation of $\sp_{2n}$. Let $\cal{O}$ be the nilpotent orbit corresponding to $\widetilde{\ell}{}$. Let
\[
\pl=\pi_1\xrightarrow{\ell_1}\pi_{2}\xrightarrow{\ell_2}\pi_{3}\xrightarrow{\ell_3}\cdots.
\]
If in each step, $\pi_{i+1}=\cal{D}^{\rm{FJ}}(\pi_i)$, i.e. $\cal{D}^{\rm{FJ}}(\pi_i)$ is irreducible, then
 we have
\[
\rm{dim}\rm{Hom}_{G_{\ge 1}^F}(\pi,\omega_{\psi })=1
\]
where $G_{\ge 1}$ is certain unipotent subgroup of $G$ corresponding to certain $F$-rational nilpotent orbit $\co_\Gamma\subset\co$.
\end{theorem}
\subsection{Wavefront sets of representations in Case (C)}

In \cite{GRS}, the wavefront sets of automorphic cuspidal representations have been calculated by the descent method, and it is corresponding to the array $\hat{\ell}(\pi)$. By the proof of Theorem 2.7 in \cite{GRS}, we can conclude that the array $\hat{\ell}(\pi)$ is actually a partition, and $\hat{\ell}(\pi)$ is of form $(2\ell_1,\cdots,2\ell_r)$. Applying the same argument on the finite fields case, we know that the $F$-stable nilpotent orbit $\wco_{\hat{\ell}(\pi)}$ is the wavefront set of $\pi$. Moreover, by replacing \cite[Lemma 2.6]{GRS} with Lemma \ref{erL3}, we get following multiplicity one theorem.

\begin{theorem} \label{main3}
Assume that $\pi$ is an irreducible cuspidal representation of $\sp_{2n}\fq$. Then $\wco_{\hat{\ell}(\pi)}$ is the wavefront set of $\pi$, and for any $F$-rational nilpotent orbit $\co\subset \wco_{\hat{\ell}{(\pi)}}$, we have
\[
\langle \pi,\Gamma_\gamma\rangle\le1
\]
where $\gamma=\{X,H,Y\}$ with $X\in \co$.
\end{theorem}

\begin{proof}
We only prove the multiplicity one part of this theorem. Let $\psi$ and $\psi'$ be two characters of $\Fq^\times$ not in the same square class.
 If the descents
$
\CD^\rm{FJ}_{\ell_1, \psi}(\pi)
$
and
$
\CD^\rm{FJ}_{\ell_1, \psi'}(\pi)
$
 of $\pi$ are either 0 or an irreducible representation in Case (C), then applying similar argument in the proof of Theorem \ref{am}, we would get the multiplicity one result. Note that the first part $\ell_1$ of $\hat{\ell}(\pi)$ is the first occurrence index $\rm{max}\{\ell^{\rm{FJ}}_{0,\psi}(\pi),\ell^{\rm{FJ}}_{0,\psi'}(\pi)\}$ of $\pi$. So the multiplicity one result follows from Theorem \ref{6c}.

\end{proof}

\subsection{Construction of $F$-rational nilpotent orbit $\co$}\label{sec8.4}
Let $\pi\in\rm{Irr}(\sp_{2n}\fq)$ be an irreducible representation in either Case (A), (B), or (C). For Case (B), we additionally assume that $\pi$ has an irreducible largest descent index. In Theorem \ref{am}, Theorem \ref{am2}, and Theorem \ref{main3}, we claim that there exist some $F$-rational nilpotent orbits $\co \subset \wco_{\hat{\ell}{(\pi)}}$ such that the irreducible representation $\pi$ appears in the GGGR of $\co$, and state some multiplicity one results. Although the construction of $\co$ has been implied in their proof, in this subsection, we shall sketch the construction $F$-rational nilpotent orbits $\co$. Suppose that $\hat{\ell}{(\pi)}=(2\ell_1,\cdots,2\ell_r)$. We construct $\co$ as follows.

(1) Calculate the descents
 \[
\CD^\rm{FJ}_{\ell_1, \psi}(\pi)=\cal{F}(\pi,\co_{((2\ell_1,1^{2(n-\ell)}),(+))},\psi)
\]
and
\[
\CD^\rm{FJ}_{\ell_1, \psi'}(\pi)=\cal{F}(\pi,\co_{((2\ell_1,1^{2(n-\ell)}),(-))},\psi),
\]
 and show that they are either irreducible representations of the same type of $\pi$ or 0, where $\psi'$ is another nontrivial additive character of $\Fq$ not in the same square class of $\psi$. This step has already been done in section \ref{sec6}.

(2) Assume that $\CD^\rm{FJ}_{\ell_1, \psi}(\pi)=\pi'\in \rm{Irr}(\sp_{2(n-\ell_1)})$ (resp. $\CD^\rm{FJ}_{\ell_1, \psi'}(\pi)=\pi'$). By induction on $n$, we can construct all $F$-rational nilpotent orbits $\co' \subset \wco_{2\hat{\ell}{(\pi')}}$ such that
     \[
     \langle \pi',\Gamma_{\gamma'}\rangle=1,
     \]
     where $\Gamma_{\gamma'}$ is the GGGRs corresponding to $\co'$.

(3) Applying Lemma \ref{erL3}, we get all $F$-rational nilpotent orbits $\co \subset \wco_{2\hat{\ell}{(\pi)}}$ such that
     \[
     \langle \pi,\Gamma_\gamma\rangle=1.
     \]

\end{document}